\pdfoutput=1
\documentclass[11pt,a4paper]{amsart}
\usepackage[margin=17mm]{geometry}

\usepackage[textsize=tiny]{todonotes}
\usepackage{enumitem}
\usepackage[colorlinks]{hyperref}

\usepackage{amssymb}
\usepackage{amsmath}
\usepackage{amsthm}
\usepackage{mathrsfs}

\usepackage[T1]{fontenc}            
\usepackage[utf8]{inputenc}        
\usepackage[polish,british]{babel}  

\usepackage{verbatim}               
\usepackage[all,line]{xy}                
\usepackage{graphicx}        
\usepackage{url}                    

\usepackage{multirow}

\def\CC{{\mathbb C}}
\def\FF{{\mathbb F}}

\def\HH{{\mathbb H}}

\def\PP{{\mathbb P}}

\def\RR{{\mathbb R}}
\def\ZZ{{\mathbb Z}}

\newcommand{\ccC}{{\mathcal{C}}}
\newcommand{\ccD}{{\mathcal{D}}}
\newcommand{\ccE}{{\mathcal{E}}}

\newcommand{\ccG}{{\mathcal{G}}}

\newcommand{\ccO}{{\mathcal{O}}}

\newcommand{\gotg}{\mathfrak{g}}

\newcommand{\ud}{\mathrm{d}}

\DeclareMathOperator{\Gr}{Gr}

\DeclareMathOperator{\Pic}{Pic}

\DeclareMathOperator{\rk}{rk}

\newcommand{\set}[1]{\left\{#1\right\}}
\newcommand{\fromto}[2]{#1, \dotsc, #2}

\newcommand{\Wedge}[1]{{\textstyle{\bigwedge\nolimits}^{\! #1}}}

\numberwithin{equation}{section}
\newtheorem{thm}[equation]{Theorem}
\newtheorem*{thm*}{Theorem}
\newtheorem{prop}[equation]{Proposition}

\newtheorem{lemma}[equation]{Lemma}

\newtheorem{conjecture}[equation]{Conjecture}

\newtheorem*{problem*}{Problem}

\theoremstyle{remark}
\newtheorem{example}[equation]{Example}
\newtheorem{rem}[equation]{Remark}

\theoremstyle{definition}

\newcounter{betweenenumi}

\theoremstyle{plain}

\renewcommand{\theenumi}{(\roman{enumi})}

\newcommand{\eemail}[1]{{\href{mailto:#1}{\nolinkurl{#1}}}}

\newcommand{\fg}{\mathfrak g}
\newcommand{\ccK}{\mathcal{K}}
\newcommand{\ccEnd}{\mathcal{E}nd}

\newcommand{\ins}{\Bigg\lrcorner}

\newtheorem{definition}[equation]{Definition}

\DeclareMathOperator{\RatCur}{RatCurves}
\DeclareMathOperator{\Stab}{\mathrm{Stab}}
\DeclareMathOperator{\LGr}{\mathrm{LGr}}
\DeclareMathOperator{\End}{End}
\DeclareMathOperator{\Hol}{Hol}
\DeclareMathOperator{\SO}{\mathbf{SO}}
\DeclareMathOperator{\CSp}{\mathbf{CSp}}
\DeclareMathOperator{\UU}{\mathbf{U}}
\DeclareMathOperator{\PGL}{\mathbf{PGL}}
\DeclareMathOperator{\SU}{\mathbf{SU}}
\DeclareMathOperator{\OOO}{\mathbf{O}}
\DeclareMathOperator{\Aut}{\mathbf{Aut}}

\DeclareMathOperator{\Gsf}{\mathbf{G}}

\DeclareMathOperator{\GL}{\mathbf{GL}}
\newcommand{\Gl}{\GL}
\DeclareMathOperator{\SL}{\mathbf{SL}}

\DeclareMathOperator{\Sp}{\mathbf{Sp}}
\DeclareMathOperator{\Spin}{\mathbf{Spin}}

\begin{document}


\title[Complex contact manifolds\dots]{Complex contact manifolds, varieties of minimal rational tangents, and exterior differential systems}


\author[J. Buczy\'nski]{Jaros\l aw Buczy\'nski}
\address{Institute of Mathematics of Polish Academy of Sciences, \\
ul. {\'S}niadeckich 8, 00-656 Warszawa, Poland,\\
         and Faculty of Mathematics, Computer Science and Mechanics of University of Warsaw, \\
         ul. Banacha 2, 02-097 Warszawa, Poland\\
         E-mail: jabu@mimuw.edu.pl}
\thanks{Buczy\'nski is supported by the Polish National Science Center project 
   Algebraic Geometry: Varieties and Structures,  2013/08/A/ST1/00804 and by a scholarship of Polish Ministry of Science.
   Moreno is supported  by the Polish National Science Centre grant under the contract number 2016/22/M/ST1/00542.}  

\author[G. Moreno]{Giovanni Moreno}
\address{Department of Mathematical Methods in Physics,  Faculty of Physics, University of Warsaw, \\
ul. Pasteura 5, 02-093 Warszawa, Poland\\
E-mail: giovanni.moreno@fuw.edu.pl}

\maketitle

\begin{abstract}
Complex contact manifolds arise naturally in differential geometry, algebraic geometry and exterior differential systems.
Their classification  would answer an important question about holonomy groups.
The geometry of such manifold $X$ is governed by the contact lines contained in $X$.
These are related to the notion of a variety of minimal rational tangents.
In this review we discuss the partial classification theorems of projective complex contact manifolds.
Among such manifolds one finds contact Fano manifolds (which include adjoint varieties) and projectivised cotangent bundles.
In the first case we also discuss a distinguished contact cone structure, arising as the variety of minimal rational tangents.
We discuss the repercussion of the aforementioned classification theorems for the geometry of quaternion-K\"ahler manifolds with positive
scalar curvature and for the geometry of second-order PDEs imposed on hypersurfaces.
\end{abstract}

\setcounter{tocdepth}{2}
\tableofcontents

\section{Introduction}\label{secIntro}

 A contact manifold can be briefly defined as
\[
\hbox{\emph{a manifold equipped with a maximally non-integrable corank-one distribution.}}
\]
Such a short, albeit correct, definition encompasses both the complex-analytic and the real-differentiable setting. In Section~\ref{secCont}
we will comment further on it. There we provide the first   examples of such objects and we state the classification problem, both from the local
and from the global point of view. Indeed, a  remarkable feature of contact manifolds is that the local equivalence problem is trivial, whereas
the global one can lead to extremely hard conjectures.

Then we switch to the real-differentiable setting and review the seemingly unrelated notions of holonomy of Riemannian manifolds
(Section~\ref{sec_holonomy}) and quaternion-K\"ahler manifolds (Section~\ref{sec_qK}). In particular, we focus on the \emph{twistor
construction}, which associates a complex contact manifold to a quaternion-K\"ahler manifold. This casts  an important  bridge between the two
worlds. The classification results concerning projective contact manifolds (reviewed in Section~\ref{sec_class}) will be mirrored by analogous
results (and conjectures) related to quaternion-K\"ahler manifolds with positive scalar curvature.
We also investigate and motivate the notion of a contact Fano manifold and the adjoint variety to a simple Lie group (Example~\ref{exAdjCont}).

In the rest of the paper we turn our attention to a feature of projective complex contact manifold, namely the fact of being  \emph{uniruled},
that is covered by rational curves.
Uniruled manifolds can be studied through their \emph{varieties of minimal rational tangents}, also known as VMRT's, briefly recalled in
Section~\ref{SecVMRT}.
For contact manifolds a minimal uniruling is made of contact lines, that is, rational curves tangent to the contact distribution and having
degree $1$.
In other words, each projective complex contact manifold is equipped with a distinguished \emph{contact cone structure} (Section~\ref{secCone})
which, save for the ``degenerate'' case of $\CC\PP^{2n+1}$, has dimension $n-1$.

As an illuminating  example of a uniruled contact manifold we choose the aforementioned {adjoint variety} to a complex simple Lie group, in which
case the corresponding VMRT is called the \emph{subadjoint variety} (Section~\ref{SecAdj}).
We conclude this paper by recalling that contact manifolds provide a natural background for second order nonlinear PDEs.
This perspective allows us to see that there is a natural procedure to associate with a contact cone structure on a contact manifold a  second order
PDE. This will be clarified in   Section~\ref{secGHom}, where we recall some recent results about   $G$-invariant  second order  PDEs, which rely
precisely on such a correspondence between contact cone structures and  second order PDEs on a contact manifold. The simplest case, that is when $G$
is of type $\mathsf{A}$, is given a particular attention, since it leads to the Monge--Amp\`ere equation (Section~\ref{secMonge}).

\subsection*{Conventions}
This survey, by vocation,  deals simultaneously with topics that range from   complex algebraic geometry through holomorphic manifolds to the real
differentiable ($C^{\infty}$) category (including quaternionic spaces) and exterior differential systems.
These distinct areas are inhabited by disparate tribes of scientists, each speaking a different  dialect of mathematical language.
Therefore we try our best to be always explicit about the objects we are working with---even  at the risk of being tedious and overloading with
adjectives. The symbol $X$ is always used to denote a complex manifold, which often will be projective (that is admitting a closed holomorphic
immersion into a complex projective space),  and hence identified with a smooth  algebraic variety over $\CC$.
The letter $M$ is going to be used for a differentiable manifold, typically endowed  with a metric $g$, that is a Riemannian manifold.

We will consider projective spaces and vector spaces over fields of real numbers $\RR$, complex numbers $\CC$,
or quaternions $\HH$.\footnote{Honestly,  $\HH$ is not a field, but this will not lead to any confusion.}
If $V\simeq \FF^{n+1}$ is a vector space over a field $\FF$, then by $\PP(V)$ (also denoted $\FF\PP^{n}$) we mean the projective space of lines
in $V$. Similarly, for a (both real and complex) vector bundle $\ccE$ over a topological space, by $\PP(\ccE)$ we mean the \emph{naive}
projectivisation of the vector bundle,  that is the set of lines in each fibre of $\ccE$.
By $V^*$ and $\ccE^*$ we mean the dual vector space and the dual vector bundle, respectively.
The notion of conormal variety $X^\#$ is deliberately used in replacement of the framework of (low-order) jets of hypersurfaces, the latter being
known only to  a niche community. Experts in PDEs imposed on hypersurfaces will find it easy to match the two languages.

A major psychological difference between a compact real (differentiable) manifold $M$ and a compact complex manifold $X$ is the abundance of
global functions:
on $M$ there are plenty of smooth functions $M \to \RR$,
while the only global holomorphic functions $X \to \CC$ are the constant ones.
This discrepancy  can be partially overcome by employing sections of line bundles on $X$ instead of functions. Indeed, sections of such bundles can
be seen as a kind of \emph{twisted} functions. This is a  crucial concept for us, since the    one-form defining a contact structure will be, by
its nature, a ``twisted'' one, see  Section~\ref{secCont}.
For a compact complex manifold $X$ and a (complex) line bundle $L$, we say that $L$ is \emph{very ample}, if it has plenty of (holomorphic!) sections.
``Plenty'' in the sense that the  natural map $X \to \CC\PP^N = \PP(H^0(X,L)^*)$ that sends  a point $x\in X$ to  the   hyperplane of  $H^0(X,L)$
made of sections vanishing on  $x$, is a closed immersion.\footnote{The real-differentiable counterpart of this property would be the fact the
points of a manifold are told apart by smooth functions:
given two points there is  always   a smooth function taking different values on them,
and given two tangent directions at a fixed point, there are two smooth functions with different corresponding partial derivatives.}

A line bundle $L$ is \emph{ample} if its power $L^{\otimes m}$ is very ample for some integer $m>0$.
Thus ample or very ample line bundles exist only on projective $X$.
In the differential geometric dialect, the ampleness is known as \emph{positivity}.\footnote{Strictly speaking,
ampleness and positivity coincide whenever   considering line bundles on complex projective manifolds.}
Explicitly, a line bundle $L$ on a K\"ahler manifold $X$ is \emph{positive} if its first Chern class $c_1(L) \in H^2(X,\ZZ)$ is represented by
a K\"ahler metric in $H^{1,1}(X, \CC) \subset H^2(X, \CC)$ (subject to the natural map $ H^2(X,\ZZ) \to H^2(X, \CC)$).

A (complex) projective manifold $X$ of complex dimension $n$ is \emph{Fano} if the top exterior power of the (complex) tangent bundle $\Wedge{n}TX$
is ample. We discuss some of many fascinating properties of Fano manifolds in Section~\ref{SecVMRT}.

It is worth stressing that, \emph{mutatis mutandis}, the main definition of a contact manifold (see Definition~\ref{defContMan}) originates and
works equally well in the real differentiable category.

\subsection*{Acknowledgements}
The authors are sincerely grateful to Aleksandra Bor{\'o}wka, Jun-Muk Hwang, Stefan Kebekus, Pawe\l\ Nurowski, Katja Sagerschnig, Travis Willse
and Jaros{\l}aw Wi{\'s}niewski,
for their numerous explanations. The authors thank the anonymous referee for carefully reading the manuscript.

\section{Complex contact manifolds}\label{secCont}
Assume $X$ is a complex manifold. By $T X$ we denote the tangent bundle to $X$.
Suppose we are given a complex vector subbundle $F \subset TX$, and denote by $L$ the quotient bundle $TX/F$. By \emph{rank} of the vector bundle $F$
we mean the complex dimension of any fibre $F_x$ for $x\in X$.
Given our motivation, that is, to define contact manifolds, at some point we will assume that $\rk F=\dim X-1$,
or equivalently,  that $L$ is a line bundle.
This should justify the notation $L$ for the quotient vector bundle, although for a while we still work in this more general setting.

Thus we have a short exact sequence of vector bundles
\begin{equation}\label{eqDefTheta}
 0 \to F \to  TX \stackrel{\theta}{\to} L \to 0.
\end{equation}
Here $\theta$ is a global section of $\Omega^1(X)\otimes L$, where $\Omega^1(X)=T^*X$ is the cotangent bundle to $X$.
That is, we can think of $\theta$ as a ``twisted'' $1$-form on $X$, or rather as a $1$-form with values in $L$.

We now pass to a local situation. That is, we pick a sufficiently small open subset $U \subset X$,
such that $L$ becomes a trivial vector bundle, $L|_U \simeq \ccO_U^{\oplus \rk L}$.
Here $\ccO_U$ denotes the trivial line bundle on $U$.
Remember, that our main case of interest is when $\rk L =1$.
The trivialisation  $L|_U\simeq \ccO_U^{\oplus \rk L}$ ``untwists'' the $1$-form $\theta$,
which now becomes a usual holomorphic form, or rather a bunch of $\rk L$ such forms, which we may conveniently  arrange into a vector
$\theta_U=(\fromto{\theta^1_{U}}{\theta^{\rk L}_{U}})$.

Accordingly, if we take the derivatives of all these local forms, which can be jointly denoted by $\ud \theta_U\in \Omega^2(U)^{\oplus \rk L}$,
then we can think of $\ud \theta_U $ as a map of vector bundles $\Wedge{2} (TU) \to L|_{U}$.
Note that here we use the inverse of the trivialisation again, and it is decisive that this trivialisation is exactly the same  as before.
Nevertheless, we stress that the map $\ud \theta_U \colon \Wedge{2} (TU) \to L|_{U}$
 strictly depends on the choice of the trivialisation $L|_{U} \simeq \ccO_U^{\oplus \rk L}$.

Now funny things start to happen.

Firstly, we may restrict $\ud \theta_U$ to $\Wedge{2} F|_{U} \subset \Wedge{2} (TU)$.
Then $(\ud \theta_U)|_{F}$ does not (!) depend on the choice of the trivialisation.
Roughly to see that, take two different trivialisations, and say,
  $A\colon \ccO_U^{\oplus \rk L} \to \ccO_U^{\oplus \rk L}$ is the ``difference'' between the two trivialisations.
Then by Leibnitz rule $\ud(A\cdot \theta_U)  = \ud A \wedge \theta_U + A \cdot \ud \theta_U$.
After the restriction to $F$, the problematic first summand vanishes, since $F = \ker \theta$.
And the twist by $A$ gets ``untwisted'' by choosing the inverse of the original trivialisations.\footnote{A reader interested in a more rigorous
proof may look, for instance, at \cite[Theorem 2.2.2]{10.2307/1970192}.}
Therefore we get a perfectly well defined map $(\ud \theta_U)|_{\Wedge{2} F|_{U}} \colon \Wedge{2} F|_U \to L|_{U}$,
  which then glues together to a well defined global map of vector bundles.
  We denote this map by
\begin{equation}\label{eqDeThetaEff}
   \ud \theta_F \colon \Wedge{2} F \to L\, ,
\end{equation}
and it is tremendously important for the content of this survey.

Secondly, the map $\ud \theta_F \colon \Wedge{2} F \to L$ is strictly related to the Lie bracket of vector fields on $X$.
Suppose $x\in X$ is a point and $\mu$ and $\nu$ are two vector fields defined near $x$.
In the simplest case, the relation to the Lie bracket is explained by the case where both $\mu$ and~$\nu$ are tangent to $F$.
Then $\ud \theta_F (\mu_x \wedge \nu_x)$ is equal\footnote{Depending on the adopted convention ``equal'' may mean ``equal up to a projective
factor'', see for instance a comment in \cite[p.~23]{jabu_dr}.} to  $\theta([\mu, \nu])_x$.
In particular:
\renewcommand{\theenumi}{(\Alph{enumi})}
\begin{enumerate}
   \item $F$ is closed under the Lie bracket (that is, $F$ is a \emph{foliation}) if and only if $\ud \theta_F$ is identically zero, and
   \item \label{item_integrable_the_form_vanishes}
         if $Y\subset X$ is a submanifold tangent to $F$ (that is, $TY \subset F|_Y$), then $\ud \theta_F |_{TY }$ is identically $0$.
\end{enumerate}
\renewcommand{\theenumi}{{\rm(\roman{enumi})}}

We are principally interested in the opposite case, that is when $F$ is as far as possible from being a foliation. The expression ``maximally
non-integrable'' appearing in the very first definition from Section~\ref{secIntro} alluded precisely to this phenomenon.

\begin{definition}\label{defContMan}
For a complex manifold $X$ with a subbundle $F\subset TX$ such that the quotient $L:=TX/F$ is a line bundle we say that
$(X,F)$ is a \emph{contact manifold} (with \emph{contact distribution} $F$) if $\ud \theta_F$ is nowhere degenerate.
Here $\ud \theta_F \colon \Wedge{2} F \to L$ is the skew symmetric bilinear map \eqref{eqDeThetaEff}  defined above from the twisted form
$\theta \colon TX \to L$, and nowhere degenerate means that for all $x\in X$ the bilinear map $(\ud \theta_F)_x \colon \Wedge{2} F_x \to L_x\simeq
\CC$ has maximal rank (equal to $\rk F$).
\end{definition}

In particular, if $(X,F)$ is a contact manifold, then
the dimension of $X$ is odd, say equal to $2n+1$,
and $\ud \theta_F$ determines an isomorphism $F \simeq F^* \otimes L$,
which we will refer to as \emph{duality}.
Moreover, $(\ud \theta_F)^{\wedge n} \colon \Wedge{2n} F \to L^{\otimes n}$ is a nowhere vanishing map of line bundles,
thus  $ \Wedge{2n} F \simeq L^{\otimes n}$ and consequently
\begin{equation}\label{eqLinBundContMan}
       \Wedge{2n+1} TX \simeq L^{\otimes (n+1)}.
\end{equation}

A few examples of a contact manifold $(X,F)$ can be easily produced.

\begin{example}\label{ex_CP_2n+1}
The odd-dimensional projective space $X=\CC\PP^{2n+1}$ is a contact manifold.
\end{example}

Indeed, one regards $\CC\PP^{2n+1}$ as the projectivisation of the $2(n+1)$ linear symplectic space $V:=\CC^{2(n+1)}$ and defines the hyperplane
$F_\ell:=\ell^*\otimes\frac{\ell^\perp}{\ell}$ in each tangent space $\ell^*\otimes \frac{V}{\ell}$ to $\PP(V)$ at $\ell\in\PP(V)$. The so-obtained
distribution of hyperplanes turns out to be a contact one.\footnote{See \cite[\S E.1]{jabu_dr} for an explicit, down to earth calculation that shows
this.} Expectedly, the line bundle $L$ turns out to be a power of the tautological bundle: $L=\ccO_{\CC\PP^{2n+1}}(2)$.

Other examples can be obtained from an $(n+1)$-dimensional manifold $Y$:

\begin{example}\label{ex_PTY}
The projectivised cotangent space $X=\PP (T^*Y)$ is a contact manifold.
\end{example}

The particular sub-case where $Y$ is the $(n+1)$-dimensional (complex) projective space $\PP^{n+1}$ will be carefully examined in the last
Section~\ref{secMonge}. A comparable example in  the real-differentiable category would be  the case $Y=\RR^{n+1} $.
Indeed, in this case,  $\PP (T^*Y)$ is the compactification of the space $J^1(n,1)$ of first-order jets of functions of $n$ variables. In some
very specific literature, such a compactification of  $J^1(n,1)$ is referred to as the space of first-order jets of hypersurfaces of $\RR^{n+1}$
\cite[Section 0.2]{MR1857908}.

Such an example of a  contact manifold $X$ provides the link with the geometric theory of partial differential equations in 1 dependent and $n$
independent variables, in both its modern declinations. In the  framework based on jets, one works with the manifold $J^1(n,1)$ and its naturally
defined contact structure \cite[Chapter 2]{MR1670044}. In the framework based on exterior differential systems, one focuses instead on the ideal
generated by the contact form \eqref{eqDefTheta} and its differential \eqref{eqDeThetaEff} in the exterior algebra of differential forms on~$X$
\cite{McKay2019,MR1083148}. Indeed, some of the remarks about contact manifolds collected in this review can be translated in terms of  geometry of
PDEs, but we leave it to the reader to find the appropriate dictionary.

Let us recall why $\PP (T^*Y)$ possesses a natural contact distribution.  Intrinsically, a~point $H\in \PP (T^*Y)$ is a hyperplane tangent to $Y$ at
the point $y=\pi(H)$, with $\pi$ being the canonical projection. Therefore, $F_H:=(\ud\pi)^{-1}(H)$ is a hyperplane tangent to $X$ at~$H$, and again
the so-obtained distribution of hyperplanes turn out to be a contact one. In this case,  $L=\ccO_{\PP (T^*Y)}(1)$.

It turns out that the contact manifolds from Examples~\ref{ex_CP_2n+1} and~\ref{ex_PTY} are all locally equivalent. Indeed, the local structure of
a (complex) contact manifold is unique up to a (holomorphic) change of coordinates in virtue of an  odd-dimensional analogue of the standard Darboux
theorem (see, for instance, \cite[Theorem 3.1]{Blair2010}).
Explicitly, for every point $x\in X$ of a contact manifold $X$, there exists an open neighbourhood $U$ of $x$, a~trivialisation of $L|_U$,
and a choice of coordinates $\fromto{\alpha_1}{\alpha_n}, \fromto{\beta^1}{\beta^n}, \gamma$ on $U$, such that
\begin{equation}\label{equ_Darboux}
  \theta|_U = \ud \gamma - \sum_{i=1}^n \alpha_i \,\ud \beta^i .
\end{equation}
By contrast, the global geometry of (compact) complex contact manifolds is really fascinating and the non-equivalent cases from
Examples~\ref{ex_CP_2n+1} and~\ref{ex_PTY} above are just the tip of the iceberg.
Before focusing on a particular  feature of it, we explain one of many motivations to study the contact manifolds, namely the
\emph{quaternion-K\"ahler manifolds}.
This motivation illustrates the origins of a family of examples: the adjoint varieties (see Section~\ref{SecAdj}).
The quaternion-K\"ahler manifolds are best explained in the context of holonomy groups, outlined in Section~\ref{sec_holonomy}.

\section{Riemannian manifolds with special holonomy}\label{sec_holonomy}

For a short while, we now leave the world of algebraic geometry and complex manifolds and take a short detour through the land of Riemannian
manifolds.
This is necessary to clarify the relationship of contact manifolds with quaternion-K\"ahler manifolds, which we review in Section~\ref{sec_qK}.

Suppose that $(M, g)$ is a Riemannian manifold of (real) dimension $m$,
  that is a differentiable manifold equipped with a Riemannian metric $g$.
Then the metric determines the notion of \emph{parallel transport\/}\footnote{Strictly speaking, a \emph{connection} determines the parallel
transport, and here we consider the Levi-Civita connection associated with $g$, see, for instance, \cite[Definition 2.2.1]{9780198506010}.
The discussion of holonomy groups can be extended to a non-metric case, for manifolds with a connection, but in this survey we restrict our
attention to the case of Riemannian manifolds.}  along piecewise smooth paths in $M$.
For a path  on $M$, that is a piecewise smooth map $\gamma\colon [0,1] \to M$, the parallel transport along $\gamma$
is a linear map $\Gamma(\gamma) \colon T_x M \to T_y M$, which can be thought of as moving the tangent vectors along the curve.
This map can be reversed  by going along the same curve in the opposite direction
and it preserves the norm and the angles between vectors.
Thus it is a linear isometry of $(T_x M, g_x)$ and $(T_y M, g_y)$.
In particular, any loop in $M$ starting and ending at $x\in M$ determines an element of the orthogonal group $\OOO(T_x M)$.
The set of all these elements is a closed subgroup of $\OOO(T_x M)$, called the \emph{holonomy group} of $M$ at $x$ and denoted by $\Hol_{x}(M)
\subseteq  \OOO(T_x M)$.

From now on we suppose in addition that $M$ is connected.\footnote{This is by no means restrictive. Indeed,  the holonomy groups of each component
are independent and, as such, they can be considered component by component.}
Then we may change the base point from $x$ to $y$ by choosing any path $\gamma$ connecting $x$ to $y$.
Then $\Hol_y(M) = \Gamma(\gamma)^{-1} \circ \Hol_x(M)  \circ \Gamma(\gamma)$,
and by choosing any linear isometry $(T_x M, g_x) \simeq (\RR^m, \langle \cdot, \cdot \rangle)$
  we obtain a subgroup $\Hol(M) \subseteq  \OOO(m)$. Here $\langle \cdot, \cdot \rangle$ denotes the standard scalar product on $\RR^m$.
Up to a conjugation in $\OOO(m)$, this subgroup is well defined, independent of the choice of $x$ or of the isomorphism $T_x M \simeq \RR^m$.

A major theorem by Berger (Theorem~\ref{thm_Berger}) shows a list of all possible holonomy groups (see \cite{10.2307/3597350} for a modern proof)
as representations in $\OOO(m)$.
For the simplicity of presentation, we restrict our attention to the case when $M$ is simply connected.
In this case, the subgroup $\Hol(M)$ is connected, as any parallel transport along a loop might be contracted to a trivial constant loop,
and the suitable homotopy (which must only use piecewise smooth loops) provides a path in $\Hol(M)$ connecting any element to the identity.
In particular, $\Hol(M) \subseteq \SO(m)$.

There are two other simplifying assumptions in the theorem of Berger.
Firstly, one assumes that $M$ is not locally isometric to a product of two Riemannian manifolds.
If $(M, g) \simeq  (M_1 ,g_1) \times (M_2, g_2)$, then $\Hol(M) = \Hol(M_1) \times \Hol(M_2) \subseteq \SO(\dim M_1) \times \SO(\dim M_2) \subseteq
\SO(m)$.
The case when the  isometry holds only locally can be dealt with similarly, though  with some additional effort. Therefore, our first  assumption can
be thought of as the \emph{irreducibility} assumption---in the sense that all possible holonomy groups can be built from the bricks that are all
listed in Theorem~\ref{thm_Berger}.

Still, there is one type of bricks  missing. This is the case of   $(M, g)$   locally isometric to a symmetric space $G/H$, where $G$ is the isometry
group of $M$ and $H$ is the subgroup of~$G$ preserving a fixed point. Then  $\Hol(M)$ is equal to $H$ and   the classification of possible holonomy
groups descends  from the classification of symmetric spaces and their isometry groups (see, for instance, \cite{9780080873244}).

Below we mention the compact \emph{real} Lie groups $\Sp(n)$, which are not to be mistaken for the \emph{complex} (or real) Lie groups
$\Sp_{2n}(\CC)$ (or $\Sp_{2n}(\RR)$) of automorphisms of $\CC^{2n}$ (or $\RR^{2n}$) preserving the standard symplectic form.
The latter appear in the process of prolongation of a contact manifold (see, for instance, \cite[Section 2.1]{GMM2018_BCP}
as well as  Section~\ref{secGHom}) and in the context of adjoint varieties (see, for instance,  \cite{Alekseevsky2019} or Section~\ref{sec_class}).
Instead, the group $\Sp(n)$ arises by regarding $\RR^{4n}$ as $\HH^n$, that is a quaternionic linear space of dimension $n$ equipped with the
standard Hermitian metric $h$.
Then
\begin{equation*}
\Sp(n)=\{ \phi\in \GL_{n}(\HH)\mid \phi^*(h)=h\}.
\end{equation*}
In fact, $\Sp(n)$ is a maximal compact subgroup of $\Sp_{2n}(\CC)$.

\begin{theorem}[M.~Berger, 1955\footnote{The original statement of the theorem is in \cite[Th\'eor\`eme 3 (p. 318)]{Berger1955}. A sketch of the
proof can be found in \cite[Section 3.4.3]{9780198506010}.
An alternative proof is provided in  \cite{10.2307/3597350}. We suggest also \cite{Bryant1998-1999}.}]\label{thm_Berger}
Let $(M,g)$ be a simply connected and irreducible Riemannian manifold which is not locally isometric to a symmetric space. Then
\begin{equation}\label{eqListHolGr}
   \Hol(M)=\SO(n), \UU(n), \SU(n),\Sp(n)\cdot\Sp(1), \Sp(n),\Gsf_2, \text{ or } \Spin(7).
\end{equation}
In each case, the representation $\Hol(M) \subset \SO(\dim M)$ is minimal and unique up to a conjugation.
\end{theorem}

The holonomy group reflects the geometric properties of the manifold.
Roughly, the Riemannian manifolds whose holonomy groups appear in \eqref{eqListHolGr}
are orientable manifolds, K\"ahler manifolds, Calabi--Yau manifolds, quaternion-K\"ahler manifolds, hyper-K\"ahler manifolds, $\Gsf_2$-manifolds and
$\Spin(7)$-manifolds, respectively.
In each case, the fact that the holonomy is a proper subgroup of $\OOO(T_xM)$ corresponds to the presence of an additional structure that is preserved
by parallel transport, that is a \emph{parallel structure}.
For instance, if the holonomy drops down to a subgroup of $\SO(n)$, it means that the manifold can be equipped with a parallel volume form.
In particular, the manifold itself should be orientable.
This simple example shows how the holonomy group can dictate restrictions on the topology of the underlying manifold.

In the remaining cases, the restrictions become more severe, starting from the dimension itself.
In the $\UU(n)$ and $\SU(n)$ cases, $M$ has to be even-dimensional.
For $\Sp(n)\cdot\Sp(1)$ and $\Sp(n)$ the dimension of $M$ must be a multiple of $4$,
   whereas in the last two cases  $\Gsf_2$ and $\Spin(7)$---called \emph{exceptional}---it has to be exactly $7$ and $8$, respectively.
On the top of that, several mutually compatible parallel structures arise.
Leaving the exceptional cases aside, they can be put in a nutshell as follows.

On a \emph{K\"ahler manifold} there is a parallel \emph{complex structure} $J$ which is also compatible with the metric,
  that is, it is an orthogonal transformation on each tangent space.
From this it follows that $M$ possesses a parallel symplectic form $\omega$ as well.
These three objects are \emph{mutually compatible} in the sense that $g(X,Y)=\omega(X,J(Y))$ for any vector fields $X$ and $Y$.

A \emph{Calabi--Yau manifold} is a K\"ahler manifold (of complex dimension $n$) with holonomy further reduced to $\SU(n)$.
As a consequence, there exists a (parallel) complex volume form, therefore the canonical line bundle $\Omega^n M$ is trivial
(as a complex line bundle).

\section{Quaternion-K\"ahler manifolds}\label{sec_qK}

A Riemannian manifold $(M,g)$ of real dimension $4n$ is called \emph{quaternion-K\"ahler}\footnote{For the sake of simplicity we skip the
case $n=1$,
where consistency requires using a different definition.} if the holonomy group $\Hol(M)$ is contained in  ${\Sp(n)\cdot \Sp(1)}$.

The scalar curvature of a quaternion-K\"ahler manifold $M$ is always constant. Being   either negative, zero, or positive,  it
can be always normalised to $-1$, $0$, or $1$, by simply rescaling the metric.
In the case of zero scalar curvature, the holonomy group is further reduced to a subgroup of $\Sp(n)$.
This case is called \emph{hyper-K\"ahler} and it is customary to assume that a quaternion-K\"ahler manifold has non-zero scalar curvature.

A typical example of a quaternion-K\"ahler manifold is the quaternionic projective space $\HH\PP^n$.
It is compact and simply connected, and its holonomy group is equal to $\Sp(n)\cdot \Sp(1)$.

There is a major difference between the cases of \emph{positive} and \emph{negative} scalar curvature.
There are no known examples of negative case that are simply connected and compact.
Negative examples with complete metric arise from (negative) \emph{Wolf spaces}  \cite{10.2307/24901319}.
Compact negative examples are obtained as quotients by a faithful action of a discrete group
(thus not simply connected).
Instead, compact positive examples are provided by Wolf spaces,
and every complete positive example is compact and simply connected \cite[Theorem~6.51]{besse_Einstein_mnflds}.
There exist examples of non-complete (in particular, non-compact) quaternion-K\"ahler manifolds $M$
with positive scalar curvature \cite{swann_phd_thesis}.
There is also a difference in the structure of twistor spaces, see below.
In \cite{Borwka2019} we find constructions of quaternionic manifolds with an action of the circle $S^1$,
which under additional assumptions on the initial data are (not necessarily compact) quaternion-K\"ahler manifolds,
with either positive or negative scalar curvature.

Each tangent space of a quaternion-K\"ahler manifold or a hyper-K\"ahler manifold possesses a triple of complex structures $I,J,K$,
which satisfy the standard quaternionic relations.
For quaternion-K\"ahler manifolds the linear subspace (of the space of linear endomorphisms of the tangent space)
generated by $I,J,K$ is parallel.
For  hyper-K\"ahler manifolds each complex structure is preserved individually by the parallel transport.

\begin{proposition}\label{prop_qK_by_I_J_K}
Let $(M, g)$ be a Riemannian manifold of dimension $4n$ with $n>1$.
Let $\ccEnd(TM) = \Omega^1 M \otimes T M$ denote the endomorphism bundle of the tangent bundle of~$M$ and let $\nabla$ be the Levi-Civita connection
on $(M,g)$.
\begin{itemize}
\item $M$ is quaternion-K\"ahler if and only if there exists a rank-three $\nabla$-invariant subbundle $\ccG \subset \ccEnd(TM)$, such that\/{\rm:}
\renewcommand{\theenumi}{{\rm(\alph{enumi})}}
\begin{enumerate}
\item  $\ccG$ is locally spanned by a triple $I,J,K$ of almost complex structures, and
\item\label{item_hermitian_with_respect_to_I_J_K}  $g$ is a Hermitian metric with respect to each $I$, $J$, and $K$, and
\item\label{item_IJ_eq_K} $IJ = K$.
\end{enumerate}
\item $M$ is hyper-K\"ahler if and only if there exists a triple $I,J,K$ of almost complex structures on $M$
satisfying {\rm\ref{item_hermitian_with_respect_to_I_J_K}} and {\rm\ref{item_IJ_eq_K}} above and  $\nabla I=\nabla J= \nabla K=0$.
\end{itemize}
\end{proposition}

A proof of both claims of Proposition \ref{prop_qK_by_I_J_K} can be found scattered throughout Sections 1.2 and 1.1 of \cite{swann_phd_thesis},
respectively.

The local almost complex structures $I,J,K$ should be interpreted analogously to the actions of elementary quaternions $i,j,k$ on $\HH$.
In particular, any element $ai + bj + ck$ with $a^2 + b^2+ c^2 =1$ is a unitary purely quaternion number and determines an embedding $\CC \subset
\HH$. Similarly, for a hyper-K\"ahler manifold $M$, any $aI + bJ + cK$ (again with norm~$1$) determines a complex structure on $M$.
If $M$ is a quaternion-K\"ahler manifold instead, usually there is no well-defined global complex structure, as the choices of $i$, $j$, $k$ vary
between the tangent spaces. For instance, $\HH\PP^n$ has no complex structure.
Instead, there is a quaternion structure, and  a well-defined sphere subbundle of $\ccG$, which consists of the $aI + bJ + cK$ with norm $1$.
This sphere bundle is called the \emph{twistor space} of $M$, and we denote it here by $X$.\footnote{In the differential geometry literature the
twistor space of $M$ is traditionally denoted by $Z$, see for instance \cite{Salamon1982}, \cite{LeBrun1994}.}

Let $X \stackrel{\pi}{\to} M$ be the natural projection.
Since $\dim M=4n$ and each fibre is two-dimensional, the real dimension of $X$ is equal to $4n+2$.
Despite $M$ usually has no complex structure,
  the non-uniqueness in the choice of $I$, $J$, $K$ in Proposition~\ref{prop_qK_by_I_J_K} is precisely resolved by going into the twistor space $X$,
  which has a natural complex structure.
Roughly, it comes from the orthogonal decomposition $T_x X = T_{m} M \oplus T_x X_m$ determined by the Levi-Civita connection $\nabla$,
   where $m =\pi(x)$ and $X_m = \pi^{-1}(m) \simeq S^2 = \CC\PP^1$.
The complex structure on $T_m M$ is given by $x \in \End(T_m M)$, while the isomorphism of~$S^2$ with $\CC\PP^1$ comes
from presenting locally the sphere bundle as a projectivisation of a (complex) vector bundle of rank two.\footnote{See \cite[Theorem~4.1]{Salamon1982}
for details. See also \cite[Section 7.1.3]{9780198506010} for the analogous construction in the case of hyper-K\"ahler manifolds.}
From now on we will view $X$ as a complex manifold of complex dimension $2n+1$.

In particular, there is a globally well-defined vector subbundle $F \subset TX$ which corresponds to choosing the horizontal part $T_{m} M$
  at every point as the orthogonal complement to the tangent space of the fibre $T_x X_m$
  (again, remember that the orthogonal splitting $T_x X = T_{m} M \oplus T_x X_m$  depends on the choice of connection
  and here we only consider the choice determined by the Levi-Civita connection).
Salamon in \cite[Theorem~4.3]{Salamon1982} shows that the subbundle $F$ is a holomorphic subbundle
and it defines a (complex) contact structure on the complex manifold $X$
(in the sense of Section~\ref{secCont}).

From now on we assume that $M$ is a \emph{positive quaternion-K\"ahler} manifold, that is a complete quaternion-K\"ahler manifold with positive scalar
curvature. In particular, $M$~is compact \cite[Theorem~6.51]{besse_Einstein_mnflds}.
In this case the twistor space $X$ admits a K\"ahler--Einstein metric of positive scalar curvature \cite[Theorem~6.1]{Salamon1982}.
In the language of algebraic geometry, this implies that the complex manifold $X$ is \emph{Fano}.
Thus from a positive quaternion-K\"ahler manifold $M$ the twistor space construction produces a contact Fano manifold $X$
   admitting a K\"ahler--Einstein metric.
Eventually, LeBrun \cite[Theorem~A]{doi:10.1142/S0129167X95000146} proved the inverse:
   if a complex contact Fano manifold admits a K\"ahler--Einstein metric, then it is the twistor space of a positive quaternion-K\"ahler manifold $M$.
Consequently:

\begin{theorem}[Positive quaternion-contact correspondence]\label{thm_qK_contact_correspondence}
The twistor construction provides a bijection between\/{\rm:}
\begin{itemize}
    \item the set of positive quaternion-K\"ahler manifolds up to homotheties $($rescalings of the metric$)$, and
    \item the set of complex contact Fano manifolds admitting a K\"ahler--Einstein metric up to biholomorphisms.
\end{itemize}
\end{theorem}

This theorem, in addition to the results mentioned above, incorporates \cite[Theorem~3.2]{LeBrun1994}.

Looking back at the statement of Theorem~\ref{thm_Berger}, a major problem throughout history was to determine whether all the possible holonomy
groups could actually be realised by a manifold $M$.
Recall that $M$ has to be looked for among manifolds that are not locally isometric to a symmetric space.\footnote{In fact, the original theorem of
Berger contained an additional  case, later dismissed as impossible.} For instance,  our prototype example $\HH\PP^n$ does not qualify as an example
of such an $M$ for the holonomy group $\Sp(n)\cdot \Sp(1)$, as this is a symmetric space with the isometry group $\Sp(n+1)/\ZZ_2$.
It turns out that all the groups listed in Theorem~\ref{thm_Berger} can be realised locally as holonomy groups of non-symmetric Riemannian manifolds.
However, it is still not known whether there exists a simply connected compact quaternion-K\"ahler manifold, which is not isometric to a symmetric
space. As shown in Theorem~\ref{thm_qK_contact_correspondence},  this problem is translated into the world of algebraic geometry.
We elaborate more on this problem in Section~\ref{sec_class}.

\section{Examples, classification statements and expectations}\label{sec_class}

Having illustrated the correspondence between complex contact manifolds and quaternion-K\"ahler geometry,
  we slowly move back to the complex contact manifolds defined and introduced in Section~\ref{secCont}.
Our first elementary examples of contact manifolds (projective space and a projectivisation of a cotangent bundle) are given above as
Examples~\ref{ex_CP_2n+1} and~\ref{ex_PTY}.
We now discuss another class of examples, namely the adjoint orbits.
They arise from complex Lie groups, such as $\SL_n(\CC)$, $\Sp_{2n}(\CC)$, $\SO_n(\CC)$ (or $\Spin_n(\CC)$) and exceptional groups. Note that these
are not the compact Lie groups discussed in the context of holonomy in Section~\ref{sec_holonomy}.
In particular, none of them is compact. The main cases of interest are the simple groups, and we restrict our presentation to those. (A reader
interested in non-compact examples of contact manifolds should also consider the semi-simple groups below.)

\begin{example}\label{exAdjCont}
   Suppose $G$ is a simple complex Lie group and denote by $\gotg$ its Lie algebra.
   Consider the adjoint action of $G$ on $\PP(\gotg)$.
   Then every odd-dimensional orbit is a contact manifold.
   In particular, the unique closed orbit $X_{\gotg}$ is a projective contact manifold called the \emph{adjoint variety} of $G$.
\end{example}

The proof is presented in \cite[\S2]{Beauville1998}.
Roughly, the contact structure comes from the symplectic form (Kostant--Kirillov form) on the orbits of the action of $G$ on $\gotg$.
More generally, any contact manifold arises from a $\CC^*$-equivariant holomorphic symplectic structure,
see \cite[\S C.5 and Theorem~E.6]{jabu_dr} for an exposition and references.

The adjoint varieties partially overlap with the previous examples.
If $G$ is of type $\mathsf{A}_{n+1}$ (that is, $G = \SL_{n+2}(\CC)$ or its finite quotient),
then the adjoint variety $X_{\gotg}$ of~$G$ is $\PP(T^*\CC\PP^{n+1})$. This is a special case of the varieties from  Example~\ref{ex_PTY} and it is
going to  be the central topic  of Section~\ref{secMonge}.
If $G$ is of type $\mathsf{C}_{n+1}$ (say, $G= \Sp_{2n+2}(\CC)$), then $X_{\gotg} = \CC\PP^{2n+1}$.
Note that the embedding $\CC\PP^{2n+1} = X_{\gotg} \subset \PP(\gotg)\simeq \CC\PP^{2n^2 +5n+2}$ is nontrivial---it is the second Veronese embedding.
Any other adjoint variety is a new example:
if $G$ is of type $\mathsf{B}_{\bullet}$ or $\mathsf{D}_{\bullet}$ (say, $G = \SO(n+4)$), then $X_{\gotg} = \Gr(\CC\PP^1, Q^{n+2})$,
that is the (complex) isotropic Grassmannian of projective lines contained in a smooth $(n+2)$-dimensional quadric hypersurface $Q^{n+2} \subset
\CC\PP^{n+3}$.
Also when $G$ is one of the (complex) exceptional groups $\mathsf{G}_2$, $\mathsf{F}_4$, $\mathsf{E}_6$, $\mathsf{E}_7$ or $\mathsf{E}_8$, then
$X_{\gotg}$ is a \emph{generalised Grassmannian}, that is a quotient of~$G$ by a maximal parabolic subgroup (corresponding to the highest root
of ~$\gotg$). See also Table~\ref{tab_list_of_adjoint}.

\begin{table}[hbt]
   \begin{center}
   \begin{tabular}{||c|c|c|c|c||}
   \hline
   \hline
      type     & $X$           & $\dim X$ & $M$ & $\ccC_x$ \\
   \hline
   \hline
      $\mathsf{A}_{n+1}$& $\PP(T^*\CC\PP^{n+1})$ & $2n+1$   & $\Gr(\CC^2, \CC^{n+2})$ &$\CC\PP^{n-1}\sqcup\CC\PP^{n-1 *}$ \\
   \hline
      $\mathsf{C}_{n+1}$& $\PP^{2n+1}$ & $2n+1$   & $\HH\PP^{n}$ & $\emptyset$\\
   \hline
      $\mathsf{B}_{\bullet}$, $\mathsf{D}_{\bullet}$ & $\Gr(\CC\PP^1, Q^{n+2})$& $2n+1$& $\widetilde{\Gr}(\RR^4, \RR^{4+n})$& 
         $\CC\PP^1\times Q^{n-2}$\\
   \hline
      $\mathsf{G}_2$ & $\mathsf{G}_2$ adjoint variety& $5$& $\mathsf{G}_2$-Wolf space & $\CC\PP^1$, twisted cubic\\
   \hline
      $\mathsf{F}_4$ & $\mathsf{F}_4$ adjoint variety& $15$& $\mathsf{F}_4$ Wolf space & $\LGr(3,6)$\\
   \hline
      $\mathsf{E}_6$ & $\mathsf{E}_6$ adjoint variety& $21$& $\mathsf{E}_6$ Wolf space & $\Gr(3,6)$\\
   \hline
      $\mathsf{E}_7$ & $\mathsf{E}_7$ adjoint variety& $33$& $\mathsf{E}_7$ Wolf space & $\mathbb{S}_6$\\
   \hline
      $\mathsf{E}_8$ & $\mathsf{E}_8$ adjoint variety& $57$& $\mathsf{E}_8$ Wolf space & $\mathsf{E}_7$-subadjoint variety\\
   \hline
   \hline
   \end{tabular}
   \end{center}
   \caption{The list of adjoint varieties $X$, corresponding Wolf spaces 
            $M$ and subadjoint varieties~$\ccC_x$ (see Section~\ref{SecAdj} for a definition and discussion of subadjoint varieties).}\label{tab_list_of_adjoint}
   
\end{table}

Since the adjoint varieties are homogeneous spaces, they are all Fano (roughly because the tangent bundle, and hence also the anticanonical line
bundle, have plenty of sections).
In fact, these are the only known contact Fano manifolds.
Moreover, they all admit K\"ahler--Einstein metrics.
That is, in accordance with Theorem~\ref{thm_qK_contact_correspondence},
   they are twistor spaces of positive quaternion-K\"ahler manifolds.
All these quaternion-K\"ahler manifolds are symmetric spaces,
   and they are called  Wolf spaces of compact simple Lie groups,
   and naturally, they exhaust all the known examples of positive quaternion-K\"ahler manifolds.
Explicitly, the Wolf spaces are the complex Grassmannian $\Gr(\CC^2, \CC^{n+2})$,  the quaternion projective space $\HH\PP^n$,
   the real Grassmannian of oriented subspaces $\widetilde{\Gr}(\RR^4, \RR^{n+4})$, and the exceptional cases.
Note that  $\Gr(\CC^2, \CC^{n+2})$ is the unique (positive) Wolf space  that has a global complex structure.
However, even in this case the twistor map  ${\PP(T^*\CC\PP^{n+1}) \to \Gr(\CC^2, \CC^{n+2})}$ is \emph{not} holomorphic.

Examples~\ref{ex_CP_2n+1}, \ref{ex_PTY}, \ref{exAdjCont} exhaust the list of known compact complex contact manifolds.
This difficulty of finding new examples, with further evidence, leads to a statement of LeBrun--Salamon conjecture.\footnote{Originally,
\cite[p.~110]{LeBrun1994} contained only a statement that  ``it is tempting to conjecture\dots'', that concerned a weaker claim about
quaternion-K\"ahler manifolds (thus, in view of Theorem~\ref{thm_qK_contact_correspondence}, with an additional assumption that $X$ is Fano and
admits a K\"ahler--Einstein metric).
Nevertheless, at least in informal discussions, the stronger claim has became known as the conjecture attributed to LeBrun and Salamon.}

\begin{conjecture}\label{conj_LeBrun_Salamon}
If $(X,F)$ is a projective complex contact manifold, then $(X,F)$ is either
\begin{itemize}
\item the projectivisation of the cotangent bundle of a projective manifold $($Example~{\rm\ref{ex_PTY})}, or
\item an adjoint variety $($Example~{\rm\ref{exAdjCont})}.
\end{itemize}
\end{conjecture}

It is also tempting to replace ``projective'' with ``compact'' in the conjecture,
   but there is not much evidence for such a claim, except perhaps in dimension $3$
   (see \cite{frantzen_peternell}, \cite{peternell_jabu_contact_3_folds},~\cite{peternell_schrack_contact_Kaehler_mflds}).
Without compactness, there are quasi-projective examples of contact manifolds resembling the adjoint varieties constructed
   in \cite{hwang_manivel_quasi_complete_contact}.
It is not clear if there is any relation of these quasi-projective examples
   to the non-complete examples of quaternion-K\"ahler manifolds in \cite{swann_phd_thesis}.

The main result towards the conjecture is the following classification theorem by Kebekus, Peternell, Sommese, Wi\'sniewski and Demailly.

\begin{theorem}\label{thm_KPSW_D}
If $(X, F)$ is a projective complex contact manifold  and $L=TX/F$ is the quotient line bundle, then  $(X,F)$ is either
\begin{itemize}
\item the projectivisation of the cotangent bundle of a projective manifold $($Example~{\rm\ref{ex_PTY})}, or
\item a projective space $($Example~{\rm\ref{ex_CP_2n+1})}, or
\item a contact Fano manifold such that $\Pic X = \ZZ \cdot [L]$ $($that is, all the complex line bundles on $X$ are isomorphic to tensor powers
of $L$ or its dual\/$)$.
\end{itemize}
\end{theorem}

Historically, \cite{4authors} shows that either $X$ is as in Example~\ref{ex_PTY}, or it is Fano and ${\Pic X \simeq \ZZ}$,
   or the canonical line bundle $K_X=\Wedge{\dim X} T^*X$ is nef.
Next \cite{demailly} excludes the last possibility.
In the case $X$ is Fano of dimension $2n+1$ and $\Pic X \simeq \ZZ$,
   since the canonical bundle $K_X \simeq (L^*)^{\otimes (n+1)}$ is divisible by $(n+1)$,
   the \emph{index} of $X$ is a positive integer equal to at most $\dim X +1 = 2n+2$, and divisible by $(n+1)$.
If the index is $2n+2$, then $X\simeq \CC\PP^{2n+2}$ by \cite{kobayashi_ochiai}.
Otherwise, the index is $n+1$ and $L$ is not divisible in the Picard group, hence $\Pic X = \ZZ \cdot[L]$.

Theorem~\ref{thm_KPSW_D} reduces Conjecture~\ref{conj_LeBrun_Salamon} to the following case:

\begin{conjecture}\label{conj_LeBrun_Salamon_for_Fano}
    If $(X,F)$ is a contact Fano manifold such that $\Pic X = \ZZ \cdot [L]$,
       then $(X,F)$ is an adjoint variety.
\end{conjecture}

This conjecture is shown for $\dim X=3$ and $\dim X =5$.
The case of $\dim X=3$ is first approached by \cite{ye}, which claims to classify projective contact threefolds---it contains convincing arguments
for Theorem~\ref{thm_KPSW_D} (for $\dim X = 3$), but not for Conjecture~\ref{conj_LeBrun_Salamon_for_Fano}.
Nevertheless, at least two approaches work: either using the classification of all Fano threefolds,
or using numerical criteria for Chern classes.\footnote{See \cite{peternell_jabu_contact_3_folds} for the latter method implemented in a more
general situation.} The case $\dim X=5$ is proved in \cite{druel}.

The low-dimensional differential geometric counterpart, that is the classification  of quaternion-K\"ahler manifolds $M$,
   or contact Fano manifolds with K\"ahler--Einstein  metric~$X$, which are their twistor spaces,
   has been shown before the projective cases. 
Explicitly, the case $\dim M=4$ is shown in \cite{hitchin}
   (recall that the definition of quaternion-K\"ahler is slightly different in the $4$-dimensional case),
   and the case $\dim M  =8$ is in~\cite{poon_salamon}.
In \cite{2herreras} the authors claimed to solve the quaternion-K\"ahler variant of Conjecture~\ref{conj_LeBrun_Salamon_for_Fano} for $\dim M =12$,
   but a mistake has been found in their statements \cite{2herreras_erratum}.
Very recently, according to the preprint \cite[Theorems~1.1 and~1.2]{jabu_wisniewski_weber_torus_on_contact},
   the quaternion-K\"ahler variant is proved for $n=3$ and $n=4$ (so $\dim M = 12$ and $\dim M = 16$).

In algebro-geometric approach, the assumption about existence of K\"ahler--Einstein metric arising from the twistor space
  is hard to digest and to apply.
There are recent interpretations in terms of $K$-stability  \cite{chen_donaldson_sun_KE_metrics_on_Fanos_1,chen_donaldson_sun_KE_metrics_on_Fanos_2,
chen_donaldson_sun_KE_metrics_on_Fanos_3}, but still they are hard to exploit in the context of contact Fano manifolds.
There are however two consequences that are very useful from the algebraic perspective:
\begin{itemize}
\item If $X$ is a complex compact manifold admitting a K\"ahler--Einstein metric, then the group of holomorphic automorphisms
of $X$ is reductive \cite{matsushima_groupe_homeomorphismes_analytiques_variete_kahlerienne}.
\item If $X$ is a complex projective manifold admitting a K\"ahler--Einstein metric,
then the tangent bundle is \emph{semistable} (in many cases this might be strengthened to \emph{stable\/})\footnote{Note
that according to \cite[Corollary~1.2]{kebekus_lines2}, the tangent bundle to a contact Fano manifold needs to be stable,
independently of the assumption about admitting a K\"ahler--Einstein metric.
However, there is a gap in this paper, see \cite[Remark~3.2]{jabu_contact_duality_and_integrability},
and this bug affects the proof of \cite[Corollary~1.2]{kebekus_lines2}.
According to loose discussion of the first named author with Kebekus,
it should be possible to fix the proof of \cite[Corollary~1.2]{kebekus_lines2}, but it needs to be carefully written down.}
\cite{lubke_stability_of_Einstein_Hermitian_vbs}.
We are not going to explain the notions of semistability and stability here.
Instead, we mention that for any ample line bundle $L$ on $X$ the following inequality of Chern classes holds
(see \cite[Theorem~0.1]{Langer_Semistable_sheaves_in_pos_char}):
\end{itemize}
\begin{equation}\label{equ_Bogomolov_Gieseker_ineq}
\bigl(2\dim X\cdot  c_2(TX) - (\dim X -1)\cdot  c_1(TX)^2\bigr)\cdot c_1(L)^{\dim X-2} \ge 0.
\end{equation}

The starting point in the classification of contact Fano manifolds $X$ of dimension $7$ or~$9$ (and the dimension of $M$ equal to $12$ or $16$)
  is the observation that the complex automorphism group of $X$ (and the   group of isometries of $M$)
  has relatively large dimension ($5$ or $8$, respectively), see \cite[Theorem~7.5]{Salamon1982}
  and \cite[Theorem~6.1]{jabu_wisniewski_weber_torus_on_contact}.
This arises from the Chern class computation that exploits in particular the inequality \eqref{equ_Bogomolov_Gieseker_ineq}
   and the Hirzebruch--Riemann--Roch Theorem.
Another ingredient is the map of sections induced from the quotient $TX \to L$, that is $H^0(TX) \to H^0(L)$.
The former vector space $H^0(TX)$ is the space of holomorphic vector fields,
   that is the Lie algebra of the automorphism group of~$X$.
The map $H^0(TX) \to H^0(L)$ is always surjective and, in the most interesting cases,
   it is in addition an isomorphism of vector spaces,
   see \cite[Proposition~1.1]{Beauville1998} and \cite[Theorem~E.13 and Corollary~E.14]{jabu_dr}.
Therefore holomorphic sections of $L$ produce families of automorphisms of $X$.
Since the aim of Conjecture~\ref{conj_LeBrun_Salamon} is to show that $X$ is homogeneous,
   in particular, we must show that it has many automorphisms, or equivalently,  that $H^0(L)$ is relatively large.
In this spirit, we have the following criteria proven throughout the last 25 years.\footnote{Further research to improve the bound on the  rank of the
torus in the condition~\ref{item_automorphisms_torus_action} of Theorem~\ref{thm_automorphisms} is ongoing by Eleonora Romano and Robert \'Smiech.}

\goodbreak
\begin{theorem}\label{thm_automorphisms}
Suppose $(X,F)$ is a contact Fano manifold, $\dim X = 2n+1$ and $L = TX/F$.
If at least one of the following conditions~{\rm\ref{item_automorphisms_toric}--\ref{item_automorphisms_dim_9}} holds, then $X$ is isomorphic to one
of the adjoint varieties.

\vglue-4pt
\vskip0pt
\begin{enumerate}
\item \label{item_automorphisms_toric}
             $X$ is a toric variety {\rm\cite{druel_toric_contact}}.
\item $\Aut(X)$ is reductive and there are sufficiently many sections of $L$,
so that the induced rational map $X\dashrightarrow \PP(H^0(L)^*)$ is generically finite onto its image {\rm\cite[Theorem~0.1]{Beauville1998}}.
\item \label{item_automorphisms_very_ample}
             $L$ is very ample {\rm\cite[Corollary~1.8 a)]{Beauville1998}}.
\item \label{item_automorphisms_torus_action}
             $\Aut(X)$ is reductive and there is an algebraic torus $(\CC^*)^r \subset \Aut(X)$ of rank $r$ such that $r \ge n-2$
             {\rm\cite[Theorem~1.3]{jabu_wisniewski_weber_torus_on_contact}}.
\item $X$ admits a K\"ahler--Einstein metric and there is an algebraic torus $(\CC^*)^r \subset \Aut(X)$ of rank $r$ such that
$r \ge \lceil\frac{n}{2}\rceil+3$ {\rm\cite[Theorem~1.1]{fang_qK_manifolds_and_symmetry_rank_2}} $($taking into account the correspondence between
contact Fano and quaternion-K\"ahler manifolds as in Theorem~{\rm\ref{thm_qK_contact_correspondence}},
see also {\rm\cite[Theorem~6.5]{jabu_wisniewski_weber_torus_on_contact})}.
\item \label{item_automorphisms_dim_9}
             $\Aut(X)$ is reductive and the dimension of $X$ is at most $9$ {\rm\cite[Theorem~1.2]{jabu_wisniewski_weber_torus_on_contact}}.
\end{enumerate}
\end{theorem}

The following example exposes one of many delicacies in the problem of deciding homogeneity of contact manifolds.

\begin{example}
   Suppose $Y$ is an abelian variety of dimension $n+1$, that is a projective manifold that has a structure of an algebraic abelian group.
   Topologically, $Y$ is a compact torus $(S^1)^{2n+2}$.
   As a complex manifold,\footnote{Note that not all sublattices $\Lambda$ will produce a projective manifold.} $Y$ is a quotient of $\CC^{n+1}$ by
   a full-rank sublattice $\Lambda \subset \CC^{n+1}$ with $\Lambda\simeq \ZZ^{2n+2}$.
   Since $Y$ has a group structure, its (co)tangent bundle is trivial $TY\simeq T^*Y \simeq \ccO_Y^{\oplus (n+1)}$.
   Thus the contact manifold $X=\PP(T^*Y)$ is isomorphic to $Y\times \CC\PP^n$.
   The automorphism group of $X$ contains $Y \times \PGL_{n+1}$, and it acts transitively on~$X$,
      thus it is similar to a homogeneous space (but it is not rational, nor Fano).
   It has many different contact structures $F \subset TX$, but for all of them $L$ is isomorphic to the pullback of $\ccO_{\CC\PP^n}(1)$.
   In particular, $L$ is nef, but not ample.
\end{example}

Similarly, among the contact manifolds obtained from Example~\ref{ex_PTY} there are quasi-homo\-ge\-neous but not homogeneous
manifolds\footnote{These are manifolds possessing  a proper open dense orbit of the automorphism group.}
obtained from some homogeneous  spaces~$Y$.
Thus it is not correct to claim that quasi-homogeneous contact manifold must be homogeneous
   (in particular, it is not known if the inverse of \cite[Corollary~1.8 b)]{Beauville1998} holds).
Instead, it is widely believed that proving Conjecture~\ref{conj_LeBrun_Salamon_for_Fano} is feasible under an additional assumption that $X$ is
quasi-homogeneous.

\section{The variety of minimal rational tangents (VMRT)}\label{SecVMRT}
In a sense, minimal rational curves are (complex) curves on a complex  manifolds that \emph{behave like lines}.\footnote{This point of view is
masterfully explained in \cite{hwang_Mori_meets_Cartan}.} This is a typical case when the complex-analytic setting parts its way from the
real-differentiable setting. Indeed, the proper generalisation of the idea of a line on a Riemannian manifold is that  of a \emph{geodesic}.
Here, the analogy is that for two sufficiently close points, there is a unique geodesic joining them---the old Euclid's fifth postulate.
In the complex-analytic setting, there is no way of measuring the ``length'' of a curve (which is a real two-dimensional object).
Instead, one can speak about its \emph{degree}.

A \emph{parametrised rational curve} in a complex projective manifold $X$ is a nonconstant generically injective morphism
\begin{equation*}
\nu_C:\CC\PP^1\longrightarrow X.
\end{equation*}
The image $C:=\nu_C(\CC\PP^1)$ is called a (unparametrised) \emph{rational curve} in $X$ and $\nu_C$ is a \emph{parametrisation} of $C$.
Indeed, the parametrisation is not unique, but there is a $ \PGL_{2}$-worth of them.

An additional structure on a complex projective manifold $X$ that is needed in order to define
the degree of a rational curve $C\subset X$ is an ample line bundle $L$.
In many cases, the line bundle is provided by an embedding of $X$ into a projective space.
Then the degree of $C$ agrees with the degree of its parametrisation,
regarded as a morphism from $\CC\PP^1$ to the projective space containing $X$.
In general, it is always a positive integer, since it coincides with the intersection number $c_1(L).C \in H^{\dim X}(X) \simeq
\ZZ\cdot[\mathrm{pt}]$.

Introducing ``the space of all rational curves on $X$'' is not as simple as in differential geometry. There, such a space is given a natural
differentiable structure\footnote{The reference textbook for such a setting is Michor's \cite{MR583436}.} and it can be conveniently studied
through the various finite-dimensional approximations provided by the jet spaces $J^k(\RR, M)$, possibly factored by the group of
reparametrisations. In the complex-analytic setting, we are not going to formalise here what ``a family of rational curves'' means.
An interested reader may check the construction of the space $\RatCur^n(X)$  of all rational curves
and the  properties of such families for  example in \cite[Section~II.2]{Kollr1996}
and, in particular, in  Definition-Proposition~II.2.11.
See also \cite[Section~2]{jabu_kapustka_kapustka_special_lines}.

$\RatCur^n(X)$ may have infinitely many connected components.

\begin{lemma}
Let $\ccK\subset \RatCur^n(X) $ be a connected component.
Then the degree of an element of $\ccK$ with respect to $L$ does not depend on the choice of the element.
\end{lemma}

See for instance \cite[Definition--Proposition 2.11]{Kollr1996}.

We say that a connected component $\ccK\subset \RatCur^n(X)$ is a \emph{uniruling} of $X$,
if there is an element $C\in\ccK$ containing a general point of $X$.
A variety is called \emph{uniruled} if it possesses a uniruling.

The simplest example of a uniruled variety is given by a ruled variety.
Every Fano manifold  is also uniruled \cite[Theorem~II.5.8]{Kollr1996}.
A Fano manifold $X$ comes with a distinguished ample line bundle $\Wedge{\dim X} TX$, the anticanonical one.

Let us see how to single out a special family of rational curves that display this aforementioned ``line-like behaviour''.
To this end, one needs an irreducible component $\ccK \subset \operatorname{RatCurve}^n(X)$ of irreducible rational curves, such that:
\begin{enumerate}
\item \label{item_minimal_curves_dominated}
$\ccK$ is a uniruling of $X$, and
\item \label{item_minimal_curves_smallest_degree}
the degree of the elements of $\ccK$ (which is well defined, as long as an ample line bundle $L$ has been chosen)
is \emph{minimal} amongst all the families fulfilling condition~\ref{item_minimal_curves_dominated}.
\end{enumerate}
This family $\ccK$, called a \emph{minimal $($with respect to $L)$ uniruling},
   to some extent behaves as the familiar set of lines in the Euclidean plane.
For instance, for a general point $x\in X$ and any point $y\in X$,
   there are at most finitely  many curves from $\ccK$, that contain $x$ and $y$.
(If in addition $y$ is sufficiently general, then there is at most one such curve.)
This is shown using a famous Mori's \emph{bend and break lemma}:
if a positive dimensional family of curves passes through two fixed points, then the family must break:
that is, in the limit we will have a reducible curve---either a multiple curve, or a union of several rational curves, see
\cite[Section~II.5]{Kollr1996}.
In particular, we would have a rational curve of lower degree, which passes through $x$, which is a general point,
hence these lower degree curves (for varying $x$) dominate $X$, contradicting item~\ref{item_minimal_curves_smallest_degree} above.

In the real-differentiable setting, the nice properties of the (infinite-dimensional) manifold of curves in $M$ come from the fact that
\emph{any curve} $C$ can be deformed in \emph{any direction}.
Indeed, the tangent space at $C$ has to be understood as the space of sections of the normal bundle $TM/TC$, and the latter always possesses sections.
In the complex-analytic settings, rational curves that can be deformed ``in any direction'' are of special interest and they are called \emph{free}.

\begin{definition}
 An element $C\in \RatCur^n(X)$ is called \emph{free} if
 \begin{equation*}
\nu_C^*( TM )=\sum_{i}\ccO_{\CC\PP^1}(a_i)\, ,\quad a_i\geq 0.
\end{equation*}
\end{definition}

Indeed, if $C$ is a smooth rational curve then from
\begin{equation*}
T_C\RatCur^n(X)=H^0(\nu_C^*( TX/TC ))
\end{equation*}
one sees that the positivity of the $a_i$ corresponds to the fact that $C$ can be deformed in any direction.
This can be generalised to the singular curves using the tangent space to the space of parametrisations.
As a consequence,

\begin{lemma}\label{lemUnirulingFreeCurve}
   An irreducible component $\ccK\subset \RatCur^n(X)$ is a uniruling if and only if $\ccK$ contains a free curve.
\end{lemma}

For a subset (for example, a uniruling) $\ccK \subset\RatCur^n(X)$ and a point $x\in X$ define $\ccK_x$ to be the set of curves that pass through $x$.
In particular, by definition, a component $\ccK$ is a uniruling if and only if for a general point $x \in X$ the set $\ccK_x$ is non-empty.
We say that a uniruling $\ccK$ is \emph{unbreakable} if for general $x\in X$ the set $\ccK_x$ is compact.
A~similar ``breaking trick'' as above can be employed to show that a  minimal uniruling is necessarily unbreakable.
Otherwise, the limit cycles would consist of several components, each of smaller degree.
Also, the same argument shows that unbreakable unirulings always exist on uniruled manifolds.

\begin{definition}\label{defVMRT}
Given a uniruling $\ccK$ on $X$, which is unbreakable (for example, minimal with respect to an ample line bundle $L$),
and a general point $x\in X$,   define the subset
\begin{equation}\label{eqDefVMRT}
      \ccC_x:=\overline{\set{ T_x C \mid C \in \ccK_x,\ C \text{ is a smooth curve}}}\subset \PP T_x X.
\end{equation}
This subset is called the \emph{variety of minimal rational tangents} (\emph{VMRT\/}) of $X$ at $x$ with respect to $\ccK$.
\end{definition}

This notion of a minimal rational tangent has no counterpart in the real-differentiable setting.
Indeed, all tangent directions can be extended to a geodesic via 
 the exponential map.
On the contrary, the existence of a nontrivial VMRT on an uniruled variety corresponds
to the fact that not all tangent directions can be ``exponentiated'' to a curve of the uniruling.

\begin{example}
If $X=\CC\PP^m$ is a projective space, then the unique unbreakable component is the Grassmannian $\Gr(\CC^2, \CC^{m+1})$ parametrising the
projective lines in $X$.
This component is minimal with respect to $L=\ccO_{\CC\PP^m}(1)$, and for any point $x\in X$ the VMRT is $\CC\PP^{m-1} = \PP(T_x \CC\PP^m)$.
\end{example}

In fact, the property that the VMRT for a general point is all the projectivised tangent space characterises $\CC\PP^{m}$ among all the (uniruled)
manifolds.
More generally, in many nice situations, the VMRT of a projective manifold $X$ at its general point can determine $X$ up to bi-holomorphism.
In \cite{hwang_Mori_meets_Cartan} this is phrased as Problem~1.3, and then is exhaustively discussed with references.
See Theorem~\ref{thm_contact_recognised_by_VMRT} for a powerful example, that only uses a small piece of the theory.

In the setting of a contact Fano manifold $(X,F)$ of dimension $2n+1$ with $\Pic(X) = \ZZ \cdot L$ (where, as usually, $L$ is the quotient $TX/F$),
the situation is more explicit.
That is, every unbreakable uniruling $\ccK$ is minimal with respect to $L$, and moreover, the degree of every curve in $\ccK$ measured by $L$ is
equal to $1$ \cite[Section~2.3]{kebekus_lines1}.
Such curves of degree~$1$ are called the \emph{contact lines}.
Note that these ``lines'' always exist, although it is not known if there exists an embedding $X \hookrightarrow \CC\PP^N$ that maps ``contact lines''
onto ordinary projective space lines.
Proving the existence of such an embedding would solve Conjecture~\ref{conj_LeBrun_Salamon}, see
Theorem~\ref{thm_automorphisms}\ref{item_automorphisms_very_ample}.

\begin{rem}
If $\pi \colon X \to M $ is the twistor space of a quaternion-K\"ahler manifold $M$ (that is, $X$ admits a K\"ahler--Einstein metric as in
Theorem~\ref{thm_qK_contact_correspondence}),
then the fibres of $\pi$ are holomorphic rational curves \cite[Theorem~14.68]{besse_Einstein_mnflds} (we stress that even though the map is not
holomorphic itself, its fibres are holomorphic curves).
The degree with respect to $L$ of each such fibre is $2$, thus from the point of view of algebraic geometry they should be called \emph{twistor
conics}, see \cite{wisniewski}.
However, in the differential-geometric literature (much more extensive), they are called \emph{twistor lines}
(see for instance \cite[p.~97]{salamon_qK_geometry} or \cite[Section~3.2.4]{borowka_phd_thesis}).
This is justified because there is a natural---though only locally defined---line bundle on some small open subsets of $X$, whose tensor square is
our $L$.
(Thus the degree with respect to this local line bundle should be half of $2$, that is $1$, which makes ``line'' a sensible notation.)
Incidentally, in the differential-geometric literature this square root of $L$ is also often denoted by letter ``$L$'' (for instance, in
\cite[from p.~148 onwards]{Salamon1982}). Confusing---is it not?
We stress that ``twistor lines/conics'' are \emph{not\/} the ``contact lines'' in the above sense. In the algebro-geometric context, it is hard to
make sense of the square-root line bundle without diving into the quaternion-K\"ahler world.
\end{rem}

Still assuming that $X$ is a contact Fano manifold with $\Pic(X) = \ZZ \cdot L$, each unbreakable $\ccK$ is compact \cite[Remark~2.2]{kebekus_lines1}.
It is expected, but not known, that there is a  unique unbreakable component.
All contact lines passing through a general point $x\in X$ are smooth \cite[Proposition~3.3]{kebekus_lines1},
but special lines might potentially be singular (again, if Conjecture~\ref{conj_LeBrun_Salamon} is true, then all contact lines are smooth as they
are just ordinary lines in a minimal homogeneous embedding of $X$).
Therefore, there is no need to use ``closure'' in the definition of VMRT (Equation~\ref{eqDefVMRT}).
A VMRT $\ccC_x \subset \PP (T_x X)$ of $X$ at a general point~$x$ is smooth \cite[Theorem~1.1]{kebekus_lines2}.

Recall (cf. \eqref{eqDeThetaEff} and Example~\ref{ex_CP_2n+1}) that the  projective space $\PP(F_x)$ is itself a contact manifold of dimension
$2n-1$,  the contact distribution on it descending from  the symplectic form $(\ud \theta_F)_x$ on $F_x$.
Moreover, the VMRT $\ccC_x$ is contained in $\PP(F_x)$ and it is \emph{Legendrian}
in this contact manifold\footnote{This means, that all the tangent spaces to $\ccC_x$ are contained in the contact distribution of $\PP(F_x)$ and
also the dimension of every component of $\ccC_x$ is equal to $n-1$.
Complex Legendrian varieties are studied in details in \cite{jabu_dr} and references therein.} \cite[Proposition~4.1]{kebekus_lines1}.
Moreover, $\ccC_x$ is not contained in any further hyperplane in $\PP(F_x)$.
Despite the claim in \cite[Theorem~1.1(2) or Theorem~3.1]{kebekus_lines2}, it is not known if $\ccC_x$ is irreducible, see
\cite[Remark~3.2]{jabu_contact_duality_and_integrability} and \cite{jabu_kapustka_kapustka_special_lines} for details of this gap
and an attempt to fix it.

If $X = X_{\gotg}$ is one of the adjoint varieties as in Example~\ref{exAdjCont}, then $\ccC_x$ does not depend on the point $x$, since the Lie
group $G$ of $\gotg$ acts on $X$ transitively and preserves the geometrically defined set of lines.
In fact, from $X$ we obtain a homogeneously embedded homogenous space $\ccC_x \subset \PP(F_x)$, where the group of automorphisms of $\ccC_x$ (and
also its representation on $F_x$) is obtained from $G$ via a simple combinatorial procedure
(see for instance \cite{landsbergmanivel02} or Section~7 of the arXiv version of \cite{jabu06}).
The homogeneous space $\ccC_x$ is called the \emph{subadjoint variety} of $G$, see the dedicated Section~\ref{SecAdj}.
Its importance is underlined by the fact that an abstract contact manifold can be recognised as an adjoint variety only by looking at its VMRT, as
a consequence of \cite[Main Theorem in Section~2]{mok_recognizing_homogeneous_manifolds}.

\begin{thm} \label{thm_contact_recognised_by_VMRT}
Suppose $(X,F)$ is a contact Fano manifold with $\Pic X = \ZZ\cdot[L]$, where $L = TX/F$.
Let $\ccK$ be an unbreakable family of contact lines and $x \in X$ a general point.
If the variety of minimal rational curves $\ccC_x$ is a homogeneous space, then $X$ is an adjoint variety of some Lie group $G$.
\end{thm}

\begin{proof}
Since Conjecture~\ref{conj_LeBrun_Salamon_for_Fano} is proved for $\dim X =3$, we may assume $\dim X\ge 5$, and thus $\dim \ccC_x \ge 2$.
By \cite[Theorem~11]{landsbergmanivel04} the VMRT $\ccC_x \subset \PP(F_x)$ must be one of the subadjoint varieties in its subadjoint homogeneous
embedding.
Thus by \cite[Main Theorem in Section~2]{mok_recognizing_homogeneous_manifolds} the manifold $X$ must be the adjoint variety.
\end{proof}

\section{Contact cone structures}\label{secCone}
Roughly speaking, a \emph{geometric structure} on a differentiable manifold $M$ is an additional structure on $M$. For instance, a tensor on $M$ can
be the source of such an additional structure. We have already come across some of them. An \emph{almost complex structure} on $M$ is a tensor
$J\in\End(TM)$, such that $J^2=-1$. A \emph{contact structure} on $M$ is a twisted\footnote{Recall Definition~\ref{defContMan}.} one-form $\theta$,
such that $d\theta$ is non-degenerate on $\ker\theta$. A~\emph{Riemannian metric} is a non-degenerate tensor $g\in S^2T^*M$, etc.

A major concern in the study of geometric structures is their \emph{problem of equivalence}, both in its local and global formulation. As we have
already pointed out, the local equivalence problem for contact structures is trivial, that is, locally there exists only one of them.
Stated differently, contact structures do not posses \emph{local differential invariants}.

The theory of differential invariants is a powerful tool in dealing with the local equivalence problem.
We refer the reader to the specialised literature to discuss the theory in more depth: a good starting point is \cite[Chapter
7]{AlekseevskyVinogradovLychagin:BIdCDG}.
References therein point to deeper and more specific works.
We just recall that the first step consists in associating a principal bundle to $M$, whose structure group reflects the geometric structure at hand.
For instance, for any point $x\in M$ of a $2n$-dimensional almost complex manifold, the subset
\begin{equation*}
\{\phi\in\GL(T_xM)\mid \phi \circ J_x=J_x\circ \phi\}\subset\GL(T_xM)
\end{equation*}
of $\RR$-linear isomorphisms preserving the complex structure $J_x$ on $T_xM$ is a subgroup isomorphic to $\GL_n(\CC)$. All these subsets, taken
together, form a $\GL_n(\CC)$-prin\-ci\-pal bundle. Similarly, for a Riemannian metric, one obtains an $\OOO_{\dim M}(\RR)$-principal bundle.
A~corank-one distribution\footnote{A contact structure is a particular case of a corank-one distribution. A key property of geometric structures,
we refrain to deepen here, is that of \emph{integrability}. In the case of distributions, the integrability in the sense of geometric structures
corresponds to the usual  integrability in the sense of Frobenius.} of a holomorphic manifold $X$ of complex dimension $m$
singles out a parabolic subgroup of $\GL_{m}(\CC)$ stabilising a hyperplane.
Such a correspondence between geometric structures and principal bundles is the cornerstone of the so-called Cartan's method of equivalence.
A standard reference for this subject is \cite{MR2003610}.

Here we focus on the geometric structures arising in the complex-analytic  setting. Roughly speaking, in such a setting, instead of equipping a
manifold with an additional set of tensors, one may consider
the common zero locus of these tensors instead and regard the latter as the \emph{correct formalisation of the concept of a geometric structure}.
The idea of employing the Cartan's method of equivalence in facing some questions arising in algebraic geometry is explained in
\cite{hwang_Mori_meets_Cartan}.
The program is based on the following definition, that can be found in Section 2 of the aforementioned paper.
Observe that, even though we present it here for a complex manifold $X$, it applies essentially verbatim  for a differentiable manifold $M$.

\begin{definition}\label{defConStruct}
A \emph{smooth cone structure} on a complex manifold $X$ is a closed nonsingular subvariety $\ccC\subset\PP(TX)$ of the projectivised tangent bundle,
such that all components of $\ccC$
have the same dimension and the natural projection map $\ccC\to X$ is a submersion. The smooth cone structure is called \emph{isotrivial} if each two
fibres $\ccC_x$ and $\ccC_y$ are isomorphic via a linear isomorphism of the corresponding tangent spaces.
\end{definition}

Note that the dimension of each fibre $\ccC_x$ for various $x\in X$ is constant and it is called the \emph{rank} of the structure.
The main example we are interested in here is the cone structure given by the VMRT of a contact Fano manifold.

\begin{example}\label{ex_VMRT_structure_for_contact_Fano}
Suppose $X$ is a contact Fano manifold with $\Pic X =\ZZ \cdot [L]$.
Then there exists an open dense subset $X_o \subset X$ such that the union $\ccC = \bigcup_{x\in X_o} \ccC_x \subset \PP(T X_o)$ is a cone structure.
If $X$ is an adjoint variety, then we may take $X_o=X$ and $\ccC$ is isotrivial.
\end{example}

The existence of the open subset $X_o$ in bigger generality is explained in \cite[Theorem~3.18]{hwang_Mori_meets_Cartan}.
In the example, we take into account the smoothness of $\ccC_x$ at general points \cite[Theorem~1.1]{kebekus_lines2}.

\begin{definition}\label{defContConStruct}
A \emph{$($smooth, isotrivial\/$)$ contact cone structure} on the contact manifold $(X,F)$ is a (smooth, isotrivial) cone structure $\ccC$ on $X$,
such that $\ccC\subset\PP(F)$.
\end{definition}

In particular, Example~\ref{ex_VMRT_structure_for_contact_Fano} shows a contact cone structure on $X_o$.
More explicitly, if $X=X_{\gotg_2}$, the adjoint variety of $\Gsf_2$, and $\ccC$ is as in Example~\ref{ex_VMRT_structure_for_contact_Fano}, then
$\ccC$ is a fibre bundle, with fibres isomorphic to twisted cubics $\CC\PP^1\subset \CC\PP^3$.
In particular, there is a subgroup of $\GL(F_x)$ preserving the fibre $\ccC_x$ isomorphic to $\Gl_2(\CC)$ and therefore there is an associated
principal $\Gl_2(\CC)$-bundle over $X$.
We will revisit this example in Section~\ref{SecAdj}.
Contact cone structures giving rise to $\GL_2$-principal bundles have found an important application in the problem of classifying  second order
PDEs, especially in relation with the notion of integrability by the method of hydrodynamic reductions \cite{MR2765729}.

We end this section by explaining why we focus on the case of a twisted cubic in dimension five. Indeed, when  $n=2$ (that is, the contact manifold
has dimension five),  the dimension of a twisted cubic $\ccC_x$ in $\PP(F_x)$ is precisely $n-1=1$.
As illustrated in Section~\ref{secGHom}, contact cone structures of rank $n-1$ on $(2n+1)$-dimensional contact manifolds are intimately related
to (scalar) second order PDEs in one dependent and $n$ independent variables.

\section{The subadjoint variety to a simple Lie group}\label{SecAdj}

The projectivised contact distribution $\PP(F)$ of a   $(2n+1)$-dimensional contact manifold $(X,F)$ displays  a curious structure: each fibre
$\PP(F_x)$ of the $\CC\PP^{2n-1}$-bundle $\PP(F)$ is a   $(2n-1)$-dimensional contact manifold on its own account. This is essentially a particular
case of Example~\ref{ex_CP_2n+1}.  As we mentioned in Section~\ref{SecVMRT}, if $(X,F)$ is   Fano with $\Pic(X)=\ZZ\cdot L$ (where $L=TX/F$), then
every unbreakable uniruling $\ccK$ of $X$ has degree 1, when measured with respect to $L$. Then, in Section~\ref{secCone} above, we have pointed out
that the VMRT $\ccC$ of $X$ with respect to such a uniruling $\ccK$ is an example of a \emph{contact cone structure}
 (see Definition~\ref{defContConStruct}).
In the present section we focus on the particular case when $(X,F)$ is the adjoint variety to a simple Lie group $G$ (see Example~\ref{exAdjCont}).

Indeed, in this case, the cone structure is isotrivial (see Example~\ref{ex_VMRT_structure_for_contact_Fano}). This basically follows from the fact
that $G$ acts transitively on $X$ and that $\ccC$ is a $G$-invariant structure. What is somewhat less evident is the fact that   $\ccC_x$ (a generic
fibre of $\ccC$) is \emph{Legendrian} in~$\PP(F_x)$. In particular, its dimension is $n-1$.

\begin{theorem}\label{thLandsBerg}
Let $X=G/P$ be the adjoint variety of the simple Lie group $G$, and let $2n+1$ be its dimension. Regard $X$ as a projective variety in $\PP(\fg)$.
Then the family of  contact lines $\ccK$ is an unbreakable uniruling for $X$ and the corresponding  VMRT $\ccC_x$ is an $(n-1)$-dimensional
\emph{Legendrian} contact cone structure on $X$.
\end{theorem}

To honour the pedagogical spirit of this survey, instead of proving Theorem~\ref{thLandsBerg} (to this end we refer the reader to
\cite{landsbergmanivel04}
and references therein), we discuss and example.  We pick the smallest exceptional simple Lie group, the 14-dimensional Lie group $\Gsf_2$. This
group  is small enough to allow a direct inspection---including its graphical illustration, see  Fig.~\ref{fig_G2}---while retaining a sufficient
complexity to render the example nontrivial. Letting $\Gsf_2$ play the ambassador role is not a novelty: we mention the excellent reviews
\cite{MR2441524}  and~\cite{bryant2000elie}. None of them, however, puts enough emphasis on the subadjoint variety.

The group $\Gsf_2$ owes its existence to the singular occurrence of a \emph{generic $3$-form} on a 7-dimensional linear space $V$, that is an element
$\phi\in\bigwedge^3V^*$, such that $\GL(V)\cdot\phi$ is open. There are various coordinate expressions for $\phi$ in the literature.\footnote{See
for instance \cite[formula (23)]{Hammerl2009} or \cite[page 12]{bryant2000elie}.} What really matters here is that the 3-form $\phi$ allows us to
define a symmetric bilinear form
\[
g(v,w):=(v\ins\phi)\wedge(w\ins\phi)\wedge\phi
\]
with values in the one-dimensional space $\bigwedge^7V$. Thanks to $g$ one can define the \emph{null quadric} $Q:=\{\ell\in\PP V\mid g|_\ell\equiv
0\}$ and the Grassmannian of the \emph{specially null planes} $X:=\{ E\in\Gr(2,V)\mid g|_E\equiv 0\, ,\phi|_E\equiv 0\}$. The group
$\Gsf_2:=\Stab_{\GL(V)}(\phi)$ is a 14-dimensional Lie group acting transitively both on $Q$ and $X$. Therefore, the latter are homogeneous spaces
of $\Gsf_2$ and they are usually denoted by $Q=\Gsf_2/P_1$ and $X=\Gsf_2/P_2$, with $P_1$ being the stabiliser of a null line and $P_2$ the stabiliser
of a specially null plane, respectively. While the minimal embedding space of $Q$ is evident, being $\CC\PP^6=\PP(V)$, the manifold $X$ embeds into
$\CC\PP^{13}=\PP(\fg_2)$, the projectivised Lie algebra of $\Gsf_2$. Both $Q$ and $X$ are 5-dimensional.\footnote{More details can be found in
\cite{bryant2000elie} and references therein.}   We do not insist here on the technical details, but on the transparent geometric construction
leading to unbreakable unirulings  on $Q$ and $X$.

\begin{figure}[htb]
    \begin{center}
        \includegraphics[width=0.8\textwidth]{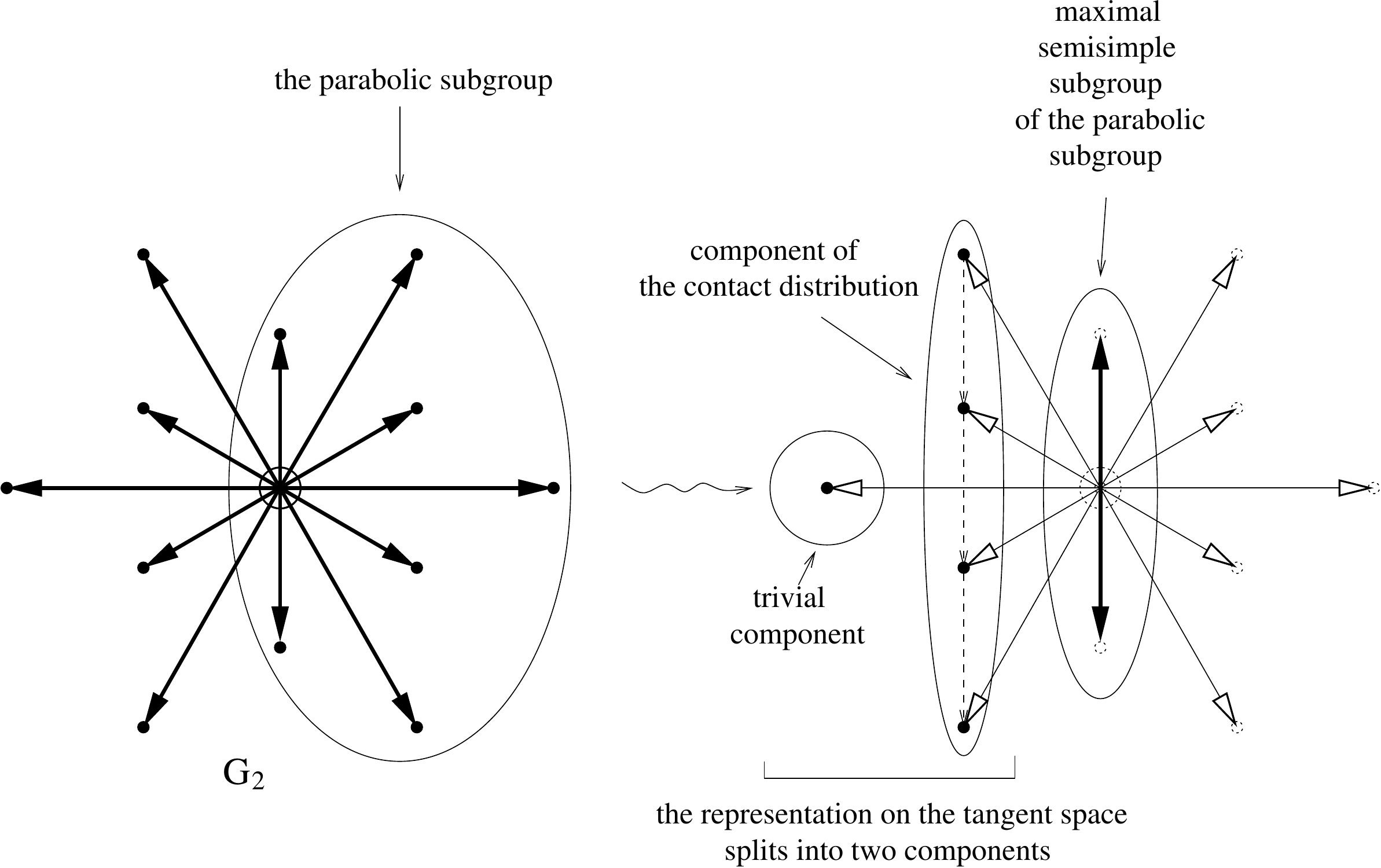}
    \end{center}
    \caption{The adjont variety of $\Gsf_2$ from the point of view of representation theory.
    In particular, the action of $\SL_2$ on $T_x X$ is visible on the right figure.}\label{fig_G2}
\end{figure}

The key is the \emph{incidence correspondence} $I:=\{ (\ell, E)\in Q\times X\mid\ell\subset E\}$, which covers both $Q$ and $X$,
\begin{equation}\label{eqIncCorrG2}
\begin{split}
\xymatrix{
 & I\ar[dr]^{p_X}\ar[dl]_{p_Q}& \\
 Q && X\ ,
}
\end{split}
\end{equation}
and turns out to be the 6-dimensional homogeneous space of $\Gsf_2$. In fact, together with  $Q$ and $X$, these are the only compact projective
homogeneous varieties of $\Gsf_2$. Both $p_Q$ and $p_X$ are $\CC\PP^1$-bundles with mutually transversal fibres. In particular,
$p_X(p_Q^{-1}(\ell))$ is a rational curve in $X$ and $p_Q(p_X^{-1}(E))$ is a rational curve in $Q$, for all $\ell\in Q$ and $E\in X$.

This obvious remark makes $\Gsf_2$ a nice testing ground for the theory of VMRTs, albeit overly simplified.
Indeed, we can regard $Q$ as a subset of $\RatCur^n(X)$ and $X$ as a subset of $\RatCur^n(Q)$, by associating with each point $\ell\in Q$ the curve
$p_X(p_Q^{-1}(\ell))$ and with each point $E\in X$ the curve $p_Q(p_X^{-1}(E))$, respectively.
The evident peculiarity of these families of rational curves, $Q$ and $X$, is that they are smooth and compact and that their dimension is the same
as the dimension of the variety they cover. As such, they contain a free element (actually, all their elements are free).
In fact, $Q$ parametrises all the lines in $X$ and consequently $Q$ is a minimal uniruling of $X$.

Although the converse does not hold, $X$ parametrises a subset of a uniruling of $Q$ by lines, and this subset is still dominant.

In the notation employed in Section~\ref{SecVMRT}, let $X_\ell$ denote the variety of all the elements of~$X$ (that are  special curves  in $Q$) that
pass through $\ell\in Q$ and by $Q_E$ the variety of all the elements of $Q$ (that are special curves in $X$) that pass through $E\in X$. By using
the same construction \eqref{eqDefVMRT} as above, we take the variety $\ccD_\ell\subset \PP(T_\ell Q)$ of all tangent directions to  the elements of
$X_\ell$ and the variety $\ccC_E\subset \PP(T_E X)$ of all tangent directions to the elements of $Q_E$. According to Definition~\ref{defVMRT},
 $\ccC_E$ is the VMRT of $X$ at $E$ \emph{with respect to} $Q$ and $\ccD_\ell$ is a subset of the VMRT of $Q$ at $\ell$
with respect to the aforementioned uniruling containing $X$.
Both are one-dimensional, but have different degrees: $\ccD_\ell$ is a line in $\CC\PP^4=\PP( T_\ell Q)$, whereas $\ccC_E$ is a rational curve of
degree~3 in $\CC\PP^4=\PP( T_E X)$. Therefore, the corresponding cones  $\hat{\ccD}_\ell$ and $\hat{\ccC}_E$ are a 2-dimensional linear subspace of
$T_\ell Q$ and a degree~3 algebraic subset of $T_E X$, respectively. In particular,  $\hat{\ccD}_\ell$ coincides with its linear span, whereas the
linear span of $\hat{\ccC}_E$ is a hyperplane $F_E$ in  $T_E X$. On the top of that, easy manipulations with the 3-form $\phi$ and the metric $g$
allow us to show that:
\begin{itemize}
\item  the distribution $\ell\longmapsto \hat{\ccD}_\ell$ is a rank-2 distribution on $Q$ with \emph{derived flag}\footnote{A $(2,3,5)$-distribution
on a five-fold can be thought of as the ``two-dimensional analogue'' of a contact distribution, that is a rank-two distribution which is maximally
nonintegrable, see \cite[Definition 2.4]{Sagerschnig2017} and references therein.} of type $(2,3,5)$;
\item the distribution $E\longmapsto F_E$ is  precisely the contact distribution of the adjoint variety $X$ of $\Gsf_2$;
\item the twisted cubic $\ccC_E$ is a \emph{Legendrian curve} in $\PP(F_E)\simeq\CC\PP^3$ with respect to the contact structure induced from the
linear  symplectic structure on $F_E$.
\end{itemize}
\looseness=-1 \noindent
It follows that the elements of the variety $Q$, regarded as a minimal (and unbreakable) uniruling of $X$,  are in fact \emph{contact lines} of $X$,
with respect to the contact distribution~$F$. Accordingly, $\ccC$ is a contact cone structure on $X$ and a particular case of those from
Example~\ref{ex_VMRT_structure_for_contact_Fano}.  The contact cone structure $\ccC$ on $X=\Gsf_2/P_2$ is, fibre by fibre, naturally isomorphic to
the \emph{subadjoint variety}
of the adjoint (contact) variety $X$ of the simple Lie group $\Gsf_2$. Theorem~\ref{thLandsBerg} has then been proved in the particular case
$G=\Gsf_2$. Notice that the role of the unbreakable uniruling $\ccK$ has been played by the smooth compact variety~$Q$.

\section{Second order PDEs associated to a contact cone structure}\label{secGHom}
In this section we explore the bridge between the theory of contact manifolds and the geometric theory of nonlinear PDEs based on exterior
differential systems (EDS).\footnote{The reader unfamiliar with EDS may take benefit from reading  McKay's gentle introduction \cite{McKay2019}. The
peculiar relationship between contact manifolds and second order nonlinear PDEs is the main subject of two recent reviews
\cite{GMM2018_BCP,MR3760967}.}  Such a bridge exists also in the complex-analytic setting, and in fact it can be obtained \emph{without} invoking
explicitly the notion of an  EDS.

From now on, $(X,F)$ is a  complex contact manifold of dimension $2n+1$. Regard $F$ as a vector bundle $F\rightarrow X$ of rank $2n$ and, for each
point $x\in X$, the form $(\ud\theta|_F)_x$ as a symplectic form on $F_x$. Observe that, due to the twist discussed in Section~\ref{secCont},
$(\ud\theta|_F)_x$ is defined up to a projective factor: it is  the object $[(\ud\theta|_F)_x]\in\PP(\bigwedge^2F_x^*)$ that is unambiguously
associated to each fibre $F_x$.
The possibility of covering $X$ by  Darboux coordinates, ascertained earlier in Section~\ref{secCont}, can be recast as follows: one can trivialise
$F$ in such a way that each trivialisation $F|_U\cong U\times\CC^{2n}$ pushes the projective class of $\ud\theta$ to the projective class of the
standard symplectic form on $\CC^{2n}$.

In other words, the bundle $F$ can be defined via $\CSp_{2n}(\CC)$-valued transition functions. Here $\CSp_{2n}(\CC)$ denotes the group of linear
transformations of $\CC^{2n}$ that preserve the projective class of  the standard symplectic form,
called the \emph{group of conformal linear symplectomorphisms}. Therefore, a $\CSp_{2n}(\CC)$-principal bundle is naturally associated to the
contact manifold $(X,F)$.

An $n$-dimensional linear subspace $L$ of $\CC^{2n}$ is called \emph{Lagrangian} if the symplectic form vanishes on it. Their collection is called
\emph{Lagrangian Grassmannian}\footnote{All relevant facts about $\LGr(n,2n)$ are reviewed in \cite{GMM2018_BCP}.} and  denoted by the symbol
$\LGr(n,2n)$. It is a nonsingular variety of dimension $\frac{n(n+1)}{2}$, naturally contained into $\Gr(n,2n)$.
Denote by $X^{(1)}\rightarrow X$ the fibre bundle, whose fibre over $x\in X$ is $\LGr(n,2n)$, and elements of the fibre represent Lagrangian
subspaces of $F_x$ with respect to the symplectic form $(\ud\theta|_F)_x$:
\begin{equation}
X^{(1)}=\{ L\mid L\textrm{ is a Lagrangian }n\textrm{-dimensional tangent space to }X\}.
\end{equation}
We shall refer to $X^{(1)}$ as the \emph{prolongation} of $X$.

Recall that an $n$-dimensional submanifold $Y\subset X$ is called \emph{Legendrian} if it is tangent to $F$.
Therefore, for any Legendrian submanifold $Y\subset X$,
we can regard the tangent bundle $T Y$  as a family of tangent spaces, that is,  elements  of $X^{(1)}$, parametrised by $Y$
(see item~\ref{item_integrable_the_form_vanishes} on page~\pageref{item_integrable_the_form_vanishes}).
In other words, we have a section $Y \to X^{(1)}|_Y$, whose image is an $n$-dimensional submanifold of $X^{(1)}$ isomorphic to $Y$.

Suppose $\ccE\subset X^{(1)}$ is a subset, such that each fibre $\ccE_x$ is a hypersurface of $X^{(1)}_x$.
We say that a ($n$-dimensional) Legendrian submanifold $Y\subset X$ is a \emph{solution} of $\ccE$ if and only if $T Y\subset \ccE$.
Accordingly, we call $\ccE$ a \emph{second order PDE} (\emph{in one dependent and $n$ independent variables}).

A skeptical reader may think in terms of the local Darboux coordinates $(\beta^i,\gamma,\alpha_i)$ on~$X$ introduced in formula~\eqref{equ_Darboux}.
Indeed, $Y\subset X$ is Legendrian if and only if there (locally) exists a function $f=f(\beta^1,\dots,\beta^n)$, such that $Y$ is described by the
equations ${\gamma=f(\beta^1,\dots,\beta^n)}$ and $\alpha_i=\frac{\partial f}{\partial \beta^i}(\beta^1,\dots,\beta^n)$, $i=1,2,\dots,n$. Then the
Darboux coordinates system can be extended to $X^{(1)}$ by introducing a set of $\frac{n(n+1)}{2}$ new coordinates $\alpha_{ij}$, with  $(i,j)$
a~pair of symmetric indices,  in such a way that $T Y$ corresponds to the additional equations $\alpha_{ij}=\frac{\partial^2 f}{\partial
\beta^i\beta^j}(\beta^1,\dots,\beta^n)$, $i,j=1,2,\dots,n$. It is then clear that a hypersurface in~$X^{(1)}$ imposes (locally) a relation among $f$
and its first and second derivatives.

Having explained how to regard hypersurfaces of $X^{(1)}$ as  second order PDEs, we produce examples, starting from an $(n-1)$-dimensional contact
cone structure on $X$, in the sense of Definition~\ref{defContConStruct}.
The core idea is classical and it can be found, for instance, in \cite[Chapter~3, Section 2]{MR2394437}. It can be grasped by looking at the diagram
\begin{equation*}
\xymatrix{
& **[r]I\subset \CC\PP^{2n-1}\times\Gr(n,2n )\ar[dr]^p\ar[dl]_q& \\
\CC\PP^{2n-1} && \Gr(n,2n).
}
\end{equation*}
Here $I$ denotes the incidence correspondence, that is the subset made of pairs $([v],L)$, with $[v]\subset L$, $v\in \CC^{2n}$ and $L\in
\Gr(n,\CC^{2n})$.

\begin{proposition}[Proposition~3.2.2 of \cite{MR2394437}]\label{propStdChow}
Let $\ccC\subset \CC\PP^{2n-1}$ be an irreducible nonsingular variety of degree $d$. Then
\begin{equation}\label{eqStdChowTrans}
\ccE_{\ccC}:=p(q^{-1}(\ccC))
\end{equation}
is an irreducible nonsingular hypersurface in $\Gr(n,\CC^{2n})$ of the same degree $d$.
\end{proposition}

The hypersurface $\ccE_{\ccC}$ defined by \eqref{eqStdChowTrans} is usually referred to as the \emph{Chow transform} of $\ccC$.
The reader may already have guessed how to exploit this Chow transform to produce second order PDEs. It suffices to adapt
Proposition~\ref{propStdChow} to the Legendrian case. Namely, one considers an \emph{entire family} (parametrised by points $x\in X$) of incidence
correspondences
\begin{equation*}
\xymatrix{
& **[r]I_x\subset \PP(F_x)\times\LGr(n,F_x)  \ar[dr]^p\ar[dl]_q& \\
\PP(F_x) && **[r]\LGr(n,F_x)=X_x^{(1)}.
}
\end{equation*}

\begin{proposition}[Lemma~23 of~{\cite{Alekseevsky2019}}]\label{propLagChow}
Let $\ccC\subset \PP(F)$ be an irreducible isotrivial contact cone structure of degree $d$. Then
\begin{equation}\label{eqLagrChowTrans}
(\ccE_{\ccC})_x:=p(q^{-1}(\ccC_x))\,
\end{equation}
is an irreducible nonsingular hypersurface in $X_x^{(1)}$ of the same degree $d$, for any $x\in X$. On the top of that, $\ccE_{\ccC}:=\bigcup_{x\in
X}(\ccE_{\ccC})_x $ is contained in $X^{(1)}$.
\end{proposition}

The hypersurface  $\ccE_{\ccC}$ can then be interpreted as a  second order nonlinear PDE imposed on the Legendrian  ($n$-dimensional) submanifolds
of $X$. Intriguingly enough, the cone structure $\ccC$ can be recovered out of $\ccE_{\ccC}$  as the set of its Cauchy characteristics.\footnote{To
clarify this would require a lengthy discussion of the characteristic variety. The reader may get an idea of such a ``reconstruction process'' from
\cite[Section 2.6]{GMM2018_BCP}.} A crucial problem in the classification of second order  PDEs for hypersurfaces of $X$ is that of finding
\emph{contact invariants}---properties of the PDEs that do not change under the action of the group of contactomorphisms of $X$. Since the
correspondence $\ccC\longmapsto \ccE_{\ccC}$ commutes with the action of this group, the contact invariants of $\ccE_{\ccC}$ can be read off from
the distribution~$\ccC$, see \cite{MR2503974} for an example of such a technique.

\section{The $n$-dimensional  Monge--Amp\`ere equation}\label{secMonge}

We apply now Proposition~\ref{propLagChow}
to the particular case of the subadjoint variety of the Lie group
$\SL_{n+2}$. The $(2n+1)$-dimensional contact manifold  $X$ will be
$\PP (T^*\CC\PP^{n+1})$, as in Example~\ref{ex_PTY}.

Accordingly, the $2n$-dimensional contact distribution $F$  on $X$ will be given by
\begin{equation}\label{eqDefFH}
F_H:=(d_H\tau )^{-1}(H),\quad H\in X,
\end{equation}
where $\tau \colon X\longrightarrow \CC\PP^{n+1}$ is the canonical
projection and $d_H\tau$ is the differential of $\tau$ at $H$. The
simple Lie group $\SL_{n+2}$ acts transitively on $X$ making it into a
homogeneous contact manifold. The subadjoint variety of  $\SL_{n+2}$,
that is the tangent directions to contact lines of $X$ (cf.
Theorem~\ref{thLandsBerg}), is peculiar amongst all the subadjoint
varieties of simple Lie groups.

It is the sole  occurrence of a reducible variety, being the union of two
$(n-1)$-dimensional projective subspaces of $\PP F$. This can be seen
by regarding $X$ as the incidence correspondence
\begin{equation}\label{eqIncCorrSL}
\aligned
\xymatrix{
& **[r]X\subset \CC\PP^{n+1}\times\CC\PP^{n+1\,\ast}\ar[dr]^q\ar[dl]_p& \\
\CC\PP^{n+1} && \CC\PP^{n+1\,\ast}}
\endaligned
\end{equation}
of pairs $(z,\pi)\in \CC\PP^{n+1}\times\CC\PP^{n+1\,\ast}$ with $z\in
\pi$.  Diagram \eqref{eqIncCorrSL} allows to recast the equation
\eqref{eqDefFH} defining $F_H$. If $H\in X$ is a tangent hyperplane to $
\CC\PP^{n+1}$ at $z\in  \CC\PP^{n+1}$, then we identify $H\equiv
(z,\pi)$, where $\pi\in \CC\PP^{n+1\,\ast}$ is the only projective
hyperplane passing through~$z$, such that $T_z\pi=H$. Accordingly,
$z=p(H)$ and $\pi=q(H)$, and \eqref{eqDefFH} reads
\begin{equation}\label{eqContDistrViaInc}
F_H=\ker d_Hp \oplus \ker d_Hq.
\end{equation}
That is, $F_H$ is spanned by the (mutually transversal,
$n$-dimensional) tangent spaces to the fibre of $p$ and to the  fibre
of $q$.  If we take fibres instead of their tangent spaces we get the union
\begin{equation*} 
p^{-1}(z)\cup  q^{-1}(\pi)\subset X
\end{equation*}
of two $n$-dimensional projective subspaces of $X$ passing through
$H$. The variety of lines passing through $H$ is
\begin{equation*}
\ccK_H=\{ \ell\subset X\mid\ell\textrm{ is a line, }  \ell\subset
p^{-1}(z)\cup  q^{-1}(\pi) ,\ \ell\ni H\} .
\end{equation*}
Observe that ``$\ell$ is a line'' makes sense since both $p$ and $q$ are
$\CC\PP^n$-bundles.

We then obtain
\begin{equation}\label{eqSubAdjCaseSL}
\ccC_H=\PP(\ker d_Hp)\cup \PP( \ker d_Hq)\equiv
\CC\PP^{n-1}\cup\CC\PP^{n-1\,\ast}.
\end{equation}
The dualisation appearing at the last projective space is due to the
fact that \eqref{eqContDistrViaInc} is a bi-Lagrangian decomposition
with respect to the symplectic form \eqref{eqDeThetaEff} allowing to
identify $ \ker d_Hq$ with the dual of $ \ker d_Hp$.

We have then provided a toy model for testing the construction mentioned
in Proposition~\ref{propLagChow}. However, before going ahead, we should
observe that  $\ccC$ is not irreducible and, as such, the aforementioned
construction cannot work on it. For reasons that will become clear
later,  we will consider only  the second irreducible component of
$\ccC$ appearing in \eqref{eqSubAdjCaseSL}, that is
$\CC\PP^{n+1\,\ast}$, denoting it by the same symbol $\ccC$.

Now we are in position of  applying Proposition~\ref{propLagChow} and
then, according to formula \eqref{eqLagrChowTrans}, we  obtain a
subset $\ccE_\ccC$ of $X^{(1)}$.

Let us describe this $\ccE_\ccC$ in the local Darboux coordinates
$(\beta^i,\gamma,\alpha_i)$, see formula~\eqref{equ_Darboux}.
First, we observe that any smooth $n$-dimensional submanifold $S\subset
\CC\PP^{n+1}$ gives rise to a (smooth, $n$-dimensional) Legendrian
submanifold $S^\#\subset X$. By the symbol $S^\#$ we mean the
\emph{conormal variety}    (see   \cite[Definition 2.1]{Russo2019}) of
$S$, which can be defined as follows:
\begin{equation*}
S^{\#}:=\{(z,\pi)\mid z\in S ,\ \pi=\mathbb{T}_zS\}.
\end{equation*}
Here by $\mathbb{T}_zS$
we mean the embedded projective tangent space at $z$ to $S$. Notice
that such a definition makes sense both in the complex-analytic and in
the smooth-differentiable setting. As an example, we remark that, in
particular, if $S$ is a projective hyperplane $\pi$, then
$\pi^\#=q^{-1}(\pi)$, that is, among these $S^\#$ there are also the
fibres of $q$.

If $S$ is locally described by $\gamma=f(\beta^1,\dots, \beta^n)$,
then  $S^\#$  will be locally described by
$\gamma=f(\beta^1,\dots,\beta^n)$ and $\alpha_i=\frac{\partial
f}{\partial \beta^i}(\beta^1,\dots,\beta^n)$, $i=1,2,\dots, n$.
This complies with the
local description of Legendrian submanifolds discussed  in Section
\ref{secGHom} above. Therefore, we can claim $f$ to be a
\emph{solution} of $\ccE_\ccC$ if $S^\#$ is a solution of $\ccE_\ccC$
in the aforementioned sense. That is, $f$ is a solution if  and only if
  $TS^\#\subset \ccE_\ccC$, which in turn, in view of formula
\eqref{eqLagrChowTrans}, means that
$\PP(T_{(z,\mathbb{T}_zS)}S^{\#})\cap\ccC_{(z,
\mathbb{T}_zS)}\neq\emptyset$ for all $z\in S$.  Equivalently, the condition
\begin{equation}\label{eqCondMA}
\dim (T_{(z,\mathbb{T}_zS)}S^{\#}\cap \ker d_Hq ) >0
\end{equation}
must be fulfilled for all $z\in S$. Now the linear space $T_{
(z,\mathbb{T}_zS)}S^{\#}$ is spanned by the vectors
\begin{equation*}
\partial_{\beta^i}+\frac{\partial^2 f}{\partial
\beta^i\partial\beta^j}(\beta^1,\ldots,\beta^n)\partial_{\alpha_j}\,
,\quad i=1,\dots,n,
\end{equation*}
where $z=(\beta^1,\dots,\beta^n,f(\beta^1,\dots,\beta^n))$ is a point
of $S$.

On the other hand, the  linear space $\ker d_Hq$ is spanned by the
$\partial_{\beta_i}$'s.

Therefore, \eqref{eqCondMA} is fulfilled if and only if the $2n\times
2n$ matrix
\begin{equation*}
\left[\begin{array}{cc}I_n & I_n \\ \left[  \frac{\partial^2 f}{\partial
\beta^i\partial\beta^j}\right]_{i,j=1}^n & 0\end{array}\right]
\end{equation*}
is degenerate, that is,
\begin{equation}\label{eqHessIndet}
\det \biggl[\frac{\partial^2 f}{\partial \beta^i\partial\beta^j}\biggr]=0
\end{equation}
for all $(\beta^1,\ldots,\beta^n)$. We have then found a local
coordinate description of $\ccE_\ccC$ which is more reminiscent of an
honest second order PDE in $n$ independent variables.

Observe that, in virtue of the naturality of our construction,  the
equation $\ccE_\ccC$:

\begin{enumerate}

\item is $\SL_{n+2}$-invariant;

\item  admits the hyperplanes of $\CC\PP^{n+1}$ as its solutions.

\end{enumerate}

On the other hand, it has been observed by various authors that   the
celebrated  ($n$-dimensional) Monge--Amp\`ere equation
\eqref{eqHessIndet} fulfills the above properties (see, e.g.,
\cite[Theorem 19]{MR2324300} for the case $n=2$ and references therein).
The fact that \eqref{eqHessIndet}  is but a local description of a more
intrinsic  geometric object, namely  $\ccE_\ccC$, having the same
properties, simply confirms these observations.

Recall that, from the very beginning, we worked on the second
irreducible component of $\ccC$ that appeared in \eqref{eqSubAdjCaseSL}.
It is then legit to wonder what would have happened, had we chosen the  other
irreducible component of $\ccC$, that is, $\CC\PP^{n+1}$. We obviously
would have obtained another equation $\ccE_\ccC$ that:

\begin{enumerate}

\item is $\SL_{n+2}$--invariant;

\item[(ii$^*$)]  admits the hyperplanes of $\CC\PP^{n+1\,\ast}$ as its solutions.

\end{enumerate}

It is not hard to convince oneself, that these are the only two
$\SL_{n+2}$-invariant second order PDEs over the homogeneous contact
manifold $X=\PP (T^*\CC\PP^{n+1})$, if by a second order PDE over $X$
one means a \emph{closed} hypersurface $\ccE\subset X^{(1)}$, such that
all the fibres $\ccE_x$, $x\in X$, are \emph{irreducible}. Indeed, this
is the simplest case dealt with by \cite[Theorem 1]{Alekseevsky2019}
(see the table column labeled by type $\mathsf{A}$), and its proof is
based on the theory of representation and the computation of the
ring of invariants (see \cite[Sections~4.3 and~6.1]{Alekseevsky2019}).

As we have already stressed at the end of the previous section,
in the theory of second order PDEs over
contact manifolds, there is a natural notion of \emph{contact
equivalence}: two PDEs $\ccE$ and $\ccE'$ are contact equivalent if and
only if there is a contactomorphism of $X$ whose lift to $X^{(1)}$ sends
$\ccE$ to~$\ccE'$ (see, e.g., \cite{MR2352610}). For instance, the
aforementioned two $\SL_{n+2}$-invariant second order PDEs are contact
equivalent by means of the so-called \emph{total Legendre transform}:
roughly speaking, one exchanges the $\alpha_i$'s with the $\beta_i$'s
(see, e.g., \cite[Example 1.2]{MR1670044}), which is just the local
consequence of fixing a (noncanonical) linear isomorphism of
$\CC\PP^{n+1}$ with its dual. Therefore, even if there are two
$\SL_{n+2}$--invariant second order PDEs over $\PP (T^*\CC\PP^{n+1})$,
they are contact equivalent. Such a peculiarity of type
$\mathsf{A}$ is also due to the  fact---already mentioned---that only
in type $\mathsf{A}$ the subadjoint variety is reducible.

In all the other types, uniqueness holds if one assumes the
degree\footnote{``Degree'' has got nothing to do with ``order'', the
latter being always two. The degree of $\ccE$ is the degree of the
hypersurface $\ccE_x$ of the Lagrangian Grassmannian $X^{(1)}_x$,
measured via its Pl\"ucker embedding, where $x\in X$ is a generic point.
In the case when $\ccE$ is equivariant, such a definition certainly
makes sense.} of the generic fibre of $\ccE$ to be \emph{minimal}
(except type $\mathsf{B}$, where the number of lowest degree invariant
PDEs is still unknown and type $\mathsf{D}$, where it is only
conjectured to be~$1$). In the aforementioned  \cite[Theorem
1]{Alekseevsky2019} the degree of such a $G$--invariant PDEs is
explicitly computed (see the first line of the table) for \emph{all
types} (including those for whose the number of invariant PDEs is still
unknown/conjectural). In particular, the lowest degree is one only in
type~$\mathsf{A}$: this reflects the fact that only in type $\mathsf{A}$
the subadjoint variety  consists of two pieces of degree one.

Indeed, the lowest degree invariant PDE does not need to coincide with
the Lagrangian Chow transform of the subadjoint variety, whose degree
tends to be very high (see the last row of  the aforementioned table of
\cite[Theorem 1]{Alekseevsky2019}). They coincide only in type
$\mathsf{G}_2$ and~$\mathsf{B}_3$ and, if each irreducible component
of the subadjoint variety is considered separately, in type
$\mathsf{A}$ as well.  The first unified construction of  a second order
$G$-invariant PDE on   $G$-homogeneous contact manifolds, for any
simple Lie group $G$, is due to D.~The. His construction led to the
Lagrangian Chow transform of the subadjoint variety and, as such, to
not necessarily minimal-degree  second order PDEs, see
\cite{The2018}. The determination of the minimal degree of a
$G$-invariant second order  PDE came a little bit later thanks to the
already cited work \cite{Alekseevsky2019}.
 
 \bibliographystyle{plain}
 \bibliography{BibUniver}

\end{document}